\documentclass[10pt,amsfonts, epsfig]{amsart}
\usepackage{amsmath, amscd, amssymb}
\usepackage{graphpap, color}
\usepackage[mathscr]{eucal}
\usepackage{mathrsfs}
\usepackage{pstricks}
\usepackage{color}
\usepackage{cancel}
\usepackage[mathscr]{eucal}
\usepackage{pstricks}
\usepackage{color}
\usepackage{cancel}
\usepackage{verbatim}

\def\bC{{\mathbf C}}

\def\bcM{\overline{\cM}}
\numberwithin{equation}{section}

\newcommand{\Spec}{\operatorname{Spec}}

\def\sO{{\mathscr O}}

\def\sL{{\mathscr L}}

\def\sO{\mathscr{O}}

\def\sF{\mathscr{F}}

\newcommand{\CC}{\mathbb{C}}
\newcommand{\EE}{\mathbb{E}}
\newcommand{\LL}{\mathbb{L}}
\newcommand{\PP}{\mathbb{P}}

\newcommand{\ZZ}{\mathbb{Z}}
\newcommand{\GG}{\mathbb{G}}
\newcommand{\FF}{\mathbb{F}}

\newcommand{\VV}{\mathbb{V}}

\newcommand{\TT}{\mathbb{T}}

\def\sK{{\mathscr K}}

\newcommand{\bk}{\mathbf{k}}

\newcommand{\bq}{\mathbf{q}}

\newcommand{\bw}{\mathbf{w}}

\newcommand{\kk}{\bk}

\newcommand{\cal}{\mathcal}
\def\ii{\mathfrak i}
\def\ee{\mathfrak e}
\def\jj{\mathfrak j}%\def\lab#1{\label{#1}[{#1}]\  }\def\cA{{\cal A}}

\def\cC{{\cal C}}
\def\cD{{\cal D}}
\def\cE{{\cal E}}
\def\cF{{\cal F}}

\def\cL{{\cal L}}
\def\cM{{\cal M}}
\def\cN{{\cal N}}
\def\cO{{\cal O}}
\def\cP{{\cal P}}
\def\cQ{{\cal Q}}

\def\cU{{\cal U}}
\def\cV{{\cal V}}

\def\cZ{{\cal Z}}

\def\fC{\mathfrak{C}}
\def\fD{\mathfrak{D}}

\def\fp{\mathfrak{p}}

\def\fu{\mathfrac{u}}
\def\fv{\mathfrak{v}}

\def\sP{{\mathscr P}}

\def\mapright#1{\,\smash{\mathop{\lra}\limits^{#1}}\,}

\def\dual{^{\vee}}

\def\sta{^\ast}

\def\virt{^{\mathrm{vir}}}
\def\upmo{^{-1}}
\def\sta{^{\ast}}

\def\mm{{\mathfrak m}}
\def\sta{^*}

\def\lra{\longrightarrow}

\def\lsta{_{\ast}}

\newcommand{\lam}{\lambda}
\newcommand{\si}{\sigma}

\def\begeq{\begin{equation}}
\def\endeq{\end{equation}}
\def\and{\quad{\rm and}\quad}
\def\bl{\bigl(}
\def\br{\bigr)}
\def\defeq{:=}

\def\sub{\subset}
\def\Ao{{\mathbb A}^{\!1}}

\def\Po{{\mathbb P^1}}
\def\and{\quad\text{and}\quad}

\def\ob{\text{ob}}

\DeclareMathOperator{\pr}{pr} 
 \DeclareMathOperator{\Ext}{Ext}
  \DeclareMathOperator{\Hom}{Hom}

 \DeclareMathOperator{\rank}{rank}
\DeclareMathOperator{\spec}{Spec}

\newtheorem{prop}{Proposition}[section]
\newtheorem{theo}[prop]{Theorem}
\newtheorem{lemm}[prop]{Lemma}
\newtheorem{coro}[prop]{Corollary}

\newtheorem{defi}[prop]{Definition}

\newtheorem{defi-prop}[prop]{Definition-Proposition}
\DeclareMathOperator{\coker}{coker}

\def\Ob{\cO b}

\def\Pn{{\mathbb P}^n}

\def\sta{^\ast}

\let\lab=\label

\def\sO{{\mathscr O}}

\def\sH{{\mathscr H}}

\def\beq{\begin{equation}}
\def\eeq{\end{equation}}

\def\vsp{\vskip5pt}
\def\Pf{{\PP^4}}

\def\bD{{\mathbf D}}

\def\fM{{\mathfrak M}}

\def\of{^{\otimes 5}}

\def\bee{\begin{equation}}
\def\eeq{\end{equation}}

\def\fA{{\mathfrak A}}

\let\eps=\epsilon

\def\ti{\tilde}

\def\barM{{\overline{M}}}

\def\fq{\mathfrak q}
\def\upf{^{\oplus 5}}

\def\fQ{{\mathfrak Q}}
\def\cpdp{\cP} % {{\cP_{g,d}^p}}

\def\cqg{{\cQ_{g}}}
\def\nq{{N}}
\def\cng{{\cN_{g}}}
\def\cngp{\cN} %{\cN_{g}^p}}
\def\Vb{\mathrm{Vb}}
\def\cpd{\cP}
\def\cpg{\cP}
\def\cpgp{\cP}% \def\cpgp{\cpdp}
\def\cvg{{\cV_g}}
\def\cvgp{\cV} %\def\cvgp{{\cV_g^p}}
\def\nn{\mathfrak n}
\def\fS{\mathfrak S}
\def\fy{\mathfrak y}

\def\fdg{{\fD_g}}
\def\tfdg{{\widetilde\fD_g}}
\def\AA{\mathbb A}
\def\lAo{_{\Ao}}
\let\lab=\label

\title{Gromov-Witten invariants of \\
stable maps with fields}%From p-field cosection localization to Gromov Witten theory of Quintic}

\date{}

\author{Huai-liang Chang and }
\address{Department of Mathematics, Hong Kong University of Science and Technology}
\email{mahlchang@ust.hk}

\author{Jun Li}
\address{Department of Mathematics, Stanford University} \email{jli@math.stanford.edu}

\begin{document}
\maketitle

\begin{abstract}
We construct the Gromov-Witten invariants of moduli of stable morphisms to $\Pf$ with fields.
This is the all genus mathematical theory of the Guffin-Sharpe-Witten model, and is 
a modified twisted Gromov-Witten invariants of $\Pf$. 
These invariants are constructed using
the cosection localization of Kiem-Li, an algebro-geometric analogue of Witten's perturbed 
equations in Landau-Ginzburg theory. We prove that these invariants coincide, up to sign, 
with the Gromov-Witten invariants of quintics.
\end{abstract}
%
%\section{Definitions}
%
%$$\fdg (fdg)
%$$
%$$\cpg (cpg); \cpgp (cpgp), \cng (cng), \cngp (cngp), \cvg (cvg), \cvgp (cvgp), \cqg(cqg)
%$$
%$$\fC (fC), \fD (fD), \Vb(Vb), \sP (sP) 
%$$

\section{Introduction}

The Candelas-dela Ossa-Green-Parkes' genus zero generating function \cite{Can} of the Gromov-Witten invariants of quintic Calabi-Yau threefolds
was proved by Givental \cite{Gi} and Lian-Liu-Yau \cite{LLY}; the genus one generating function
of Bershadsky-Cecotti-Ooguri-Vafa's \cite{BCOV} was proved by Zinger \cite{Zi}.
%The genus one generating function is proved by  following the genus one hyperplane property of 
Both proofs rely on the ``hyperplane property'' of the Gromov-Witten invariants of quintics,
which expresses the invariants in terms of ``Euler class of bundles'' over the moduli of stable morphisms to $\Pf$. 
The hyperplane property for genus zero was derived by Kontseviech \cite{Ko}; the case of genus one was proved by
Li-Zinger \cite{LZ}. This paper is our first step to build such a theory for all genus Gromov-Witten invariants
of quintics, and beyond.

In this paper, we introduce a new class of moduli spaces:
the moduli of stable morphisms to $\Pf$ with fields. These moduli spaces are  cones over
the usual moduli of stable morphisms to $\Pf$; they are not proper for positive genus. We use Kiem-Li's cosection localized
virtual cycle to construct their localized virtual cycles, thus their Gromov-Witten invariants. Applying degeneration, we prove
that these invariants coincide (up to signs) with the Gromov-Witten invariants of the quintics.

%The moduli space of stable morphisms with $p$-fields
%is an algebro-geometric analogue and generalization of the ``Guffin-Sharpe-Witten" model,
%the later is the genus zero Landau-Ginzburg model introduced  
%by Guffin and Sharpe in \cite{Sharpe}, following Witten's Gauged-Linear-Sigma-Model program 
%\cite{GLSM}. 
%By viewing this moduli space as a cone over the moduli of stable morphisms to $\Pf$, 
%and viewing its lcoalized virtual cycle
%as a generalized ``Euler class of bundle'', it answers the question posted in the beginning of this paper.

We briefly outline our construction and the main theorem. Given non-negative
integers $g$ and $d$, we form
the moduli $\bcM_g(\Pf,d)^p$ of genus $g$ degree $d$ stable morphisms to $\Pf$
%. It parameterizes
%all morphisms $u: C\to \Pf$ from nocal (connected) arithmetic genue $g$ curves $C$ to $\Pf$ of degree $d$. 
%such that the automorphism group of $u$ is finite. The moduli of stable morphisms to $\Pf$ 
with $p$-fields:
$$\bcM_g(\Pf,d)^p=\{[u,C,p]\ \big| \ [u,C]\in \bcM_g(\Pf,d),\ p\in \Gamma(C, u\sta \sO_\Pf(-5)\otimes \omega_C)\,\}/\sim.
$$
Here $\bcM_g(\Pf,d)$ is the moduli of degree $d$ genus $g$ stable morphisms to
$\Pf$. %, and $\sO(-5)=\sO_{\Pf}(-5)$. 

It is a Deligne-Mumford stack; forgetting the fields, the induced morphism
$$\bcM_g(\Pf,d)^p\to\bcM_g(\Pf,d)
$$ 
has fiber $H^0(u\sta \sO_\Pf(-5)\otimes\omega_C)$
over $[u,C]\in\bcM_g(\Pf,d)$.
When $g$ is positive, it is not proper. 

The moduli space $\bcM_g(\Pf,d)^p$ has a perfect obstruction theory, thus has a virtual class. 
To overcome its non-properness in order to define its Gromov-Witten invariant,
we construct a cosection
(homomorphism) of its obstruction sheaf. The choice of the cosection
depends on the choice of a degree five homogeneous
polynomial, like $\bw=x_1^5+\ldots+x_5^5$.
The non-surjective loci (called the degeneracy loci) of the cosection associated to $\bw$
$$\sigma : \Ob_{\bcM_g(\Pf,d)^p}\lra \sO_{\bcM_g(\Pf,d)^p}
$$
is
%of $\sigma$ is exactly the  moduli of stable morphisms to $Q$: 
$$%\text{the non-surjective loci of\,}\sigma=
\bcM_g(Q,d)\sub \bcM_g(\Pf,d)^p, \quad Q=(x_1^5+\ldots+x_5^5=0)\sub \Pf,
$$
which is proper.
Applying Kiem-Li cosection localized virtual class construction, we obtain a localized virtual cycle
$$[\bcM_g(\Pf,d)^p]\virt_\sigma\in A_0 \bcM_g(Q,d).
$$
%(One checks that the cycle is of pure dimension zero, thus lies in $A_0$.) We define
We define the Gromov-Witten invariant of $\bcM_g(\Pf,d)^p$ be
$$N_g(d)^p_\Pf=\deg [\bcM_g(\Pf,d)^p]\virt_\sigma.
$$
(We also call them the Gromov-Witten invariants of the space $(K_{\Pf},\bw)$.) %, to be explained later.)

\vsp
It relates to the Gromov-Witten invariants the quintic $Q$:
%$N_g(d)_Q=\deg [\bcM_g(Q,d)]\virt$ of 

\begin{theo}\label{thm1.1}
For $g\ge 0$ and $d>0$,
the Gromov-Witten invariant of $\bcM_g(\Pf,d)^p$ (or $(K_\Pf,\bw)$) %of $\bcM_g(\Pf,d)^p$ 
coincides with the Gromov-Witten 
invariant $N_g(d)_Q$ of the quintic $Q$ up to a sign:
$$%[\bcM_g(\Pf,d)^p]\virt_\sigma
N_g(d)^p_\Pf=(-1)^{5d+1-g}N_g(d)_Q.
$$
\end{theo}

When $g=0$, this is derived in Guffin-Sharpe \cite{Sharpe} using path-integral. This identity also is
the Kontsevich's formula on $g=0$ Gromov-Witten invariants of quintics.
If one views the localized virtual cycle of $\bcM_g(\Pf,d)^p$ as ``Euler class of bundles'',
this theorem is a substitute of the ``hyperplane property'' of the Gromov-Witten invariants
of quintics in high genus.

%as a generalized ``Euler class of bundle'',

%Since $\bcM_0(\Pf,d)^p=\bcM_0(\Pf,d)$ is smooth, and its obstruction sheaf
%is $R^1\pi\lsta\bl f\sta\sO(-5)\otimes\omega_{\cC/\bcM_0(\Pf,d)}\br$, 
%where $(f,\pi,\cC)$ is the universal family of $\bcM_0(\Pf,d)$, the localized virtual cycle theory \cite{KL}
%implies
%$$[\bcM_0(\Pf,d)^p]\virt_\sigma=c_{5d+1}\bl R^1\pi\lsta\bl f\sta\sO(-5)\otimes\omega_{\cC/\bcM_0(\Pf,d)}\br
%[\bcM_0(\Pf,d)].
%$$
%By Serre duality, we obtain $N_0(d)^p=(-1)^{5d+1}N_0(d)$,
%the Kontsevich's formula expressing genus zero quintic Gromov-Witten invariants as the ``Euler class of bundles''.
%Here we use that $R^1\pi\lsta\bl f\sta\sO(-5)\otimes\omega_{\cC/\bcM_0(\Pf,d)}\br$ is the dual of
%$\pi\lsta f\sta\sO(5)$. 

\vsp
We believe this construction will lead to a mathematical approach to Witten's Gauged-Linear-Sigma model 
%(in short GLS model) 
for all genus. 
In \cite{GLSM}, Witten constructed a (gauged) topological field theory (for $g=0$) 
whose target is the stacky quotient $[\CC^6/\CC^\ast]$ (of weights $(1,1,1,1,1,-5)$)
with a superpotential, say $\bw$.
This theory has two GIT quotients: one is $(K_{\Pf},\bw)$, called the massive theory; the other is $([\CC^5/\ZZ_5],\bw)$.
%The superpotential on the Artin stack descends to superpotential of 
%$K_{\Pf}$ given by the quintic polynomial $w_\Pf$ and also descends to superpotential on 
%$[\CC^5/\ZZ_5]$ given by the homogenious quintic polynomial 
%where $\bw=\sum x_i^5$, 
called the linear Landau-Ginzberg model.\footnote{Linear Landau-Ginzburg model means 
the space is the orbifold quotient of an affine space.}
%The Gauged Linear Sigma model of 
Witten proposed to A-twist both models: the A-twist
of $(K_{\Pf},\bw)$ likely is a theory of moduli of
stable quotients, and the resulting theory is of Landau-Ginzburg type. 
%It is named massive LG model for $(K_{\Pf},w_\Pf)$ because the instantons it enumerates can be massive. 
The $A$-twist of $([\CC^5/\ZZ_5],\bw)$ is 
%a generalization of the topological gravity coupled with matter 
related to the generalized Witten conjecture \cite{ALG} for $A_{4}=(\CC,x^5)$.
 
%Witten suggested that ``after integrating out the massive instanton in the massive 
%LG model for $(K_{\Pf},\bw)$, one obtains the GW theory of 
%Quintic threefold''. Super-String theories also speculate that ``the topological gravity coupled with 
%matter fits in the matrix model and the mirror symmetry between
%$([\CC^5/\ZZ_5],\bw)$ and its mirror should easier to workout''. 

%Putting together, Witten's GLSM program initiated in 
The program proposed in \cite{GLSM} provides a possible road
map towards an all genus mathematical theory linking the Gromov-Witten theory of quintic 
to the Landau-Ginzberg model of $([\CC^5/\ZZ_5],\bw)$. A bolder speculation is that there is 
a geometric mirror construction identifying the
A-twisted topological string theory of $([\CC^5/\ZZ_5],\bw)$ with the B-side invariants of its 
Landau-Ginzburg Mirror. % whose definition for $g>0$ remains open in mathematics.
 
In \cite{FJRW}, Fan, Jarvis and Ruan constructed
the virtual cycle of the $A$-twisted topological string theories of the linear Landau-Ginzberg model of $([\CC^5/\ZZ_5],\bw)$;
their construction is via analytic purtubation of Witten's equation. 
%This provides a symplectic geometric, and hence mathematical, construction of the theory of all genus. The construction of \cite{FJRW} is based on , which is known to be the symplectic geometric counterpart of an A-twisted path integral of Landau-Ginzburg type (after Mathai-Quillen formalism).  
Later,  Ruan and Chiodo proved  \cite{Ruan} the genus zero mirror symmetry for $([\CC^5/\ZZ_5],\bw)$ and its mirror.
  
For massive theory of $(K_{\Pf},\bw)$, Marian, Oprea and Pandharipande
constructed the moduli of stable quotients \cite{StableQ}, which is believed to be an example of massive 
instantons. It is interesting to see how the invariants of $A$-twisting the construction in \cite{StableQ} 
relate to the invariants of the massive instantons in $(K_{\Pf},\bw)$ in Witten's program. 
%However the virtual fundamental classes of Landau Ginzburg type in the theory still needs to be constructed.
    
Using Super-String theories, Guffin and Sharpe constructed a special type of genus zero Landau-Ginzberg model for 
$(K_{\Pf},\bw)$, and equated it with the genus zero
Gromov-Witten invariants of the quintic $Q$ \cite{Sharpe}. 
%It is a massless theory of Landau Ginzburg type which should be the resulting theory that comes out from the massive theory of $(K_{\Pf},\bw_\Pf)$ by integrating out all massive instantons. 
%In \cite{Sharpe} Guffin-Sharpe use a path integral to define the theory.  The proof to show such integral gives GW theory of quintic in \cite{Sharpe} is based on path integral techinique 
%and hence is in the realm of infinite dimensional differential geometry. 
%There should be an algebraic geometric construction of this Guffin-Sharpe-Witten model for all genus and a proof equating it with GW invariant of quintic.
The notion of $p$-fields was introduced in this work.
Using non-perturbative localization of path-integral, they reduced this theory to the genus zero Gromov-Witten invariants
of quintics. Since this follows Witten's Gauged-Linear-Sigma-Model program, we call this construction the Guffin-Sharpe-Witten model. % (in short GSW model) in this paper.
%This $p$-field arises in \cite{Sharpe} 
%to be the field one gains after $A$-twisted a smooth map $\bar{f}:C\to K_\Pf$ which
%project to $f:C\to \Pf$. Namely $p$-field is a priori a smooth section of $\Gamma(C,f^\ast K_\Pf)$ but after twist it lies in $\Gamma(C,f^\ast(\sO(-5)\otimes \omega_C))$.
\vsp

Our work is an algebro-geometric construction of Guffin-Sharpe-Witten model for all genus. 
The moduli of stable morphisms with $p$-fields is the algebro-geometric substitute of the
phase space of all smooth maps with smooth fields. The cosection localized virtual cycle 
is the analogue of Witten's perturbed equation. 
Theorem \ref{thm1.1} shows that the Gromov-Witten invariants of the algebro-geometric Guffin-Sharpe-Witten model
of all genus coincide up to signs with the Gromov-Witten invariants of 
quintic threefolds.

Our construction applies to global complete intersection Calabi-Yau threefolds of toric varieties. 
In the subsequent papers, we will apply the techniques developed to
the moduli of stable quotients (cf. \cite{StableQ}) to obtain all genus invariants of 
massive theory of $(K_\Pf,\bw)$ \cite{CL-1}; we will also apply it to the linear Landau-Gingzberg model to obtain an alternative
algebro-geometric construction of Fan-Jarvis-Ruan-Witten invariants \cite{CLL} . In the later case, 
the resulting invariants are equal to those defined using perturbed the Witten equations \cite{FJRW}.

We believe the new invariants and their equivalence with the Gromov-Witten invariants
of quintics provide the first step toward building a geometric bridge establishing
the conjectural equivalence of Gromov-Witten invariants
of quintics and the Fan-Jarvis-Ruan-Witten invariants of $([\CC^5/\ZZ_5],\bw)$. 
Constructing such bridge will be the long term goal of this project.
%{\red [I do not understand this part] It provides an explicit link between the GW theory of the quintic $Q$, via GSW massless model of $(K_\Pf,\bw_\Pf)$, to the missing massive model of $(K_\Pf,\bw_\Pf)$ in GLSM program, by tracing the creation and annihilation of massive intstantons. From this point of view we approach GLSM program from the Calabi Yau side, compared to the work of \cite{FJRW} to define FJRW theory from the (orbifold) Landau Ginzburg side $([\CC^5/\ZZ_5],\bw)$.}

\vsp
\noindent
{\bf Acknowledgement}. The first author thanks E. Sharpe for his lecture on topological field theory in the university of Utah, summer 2007.
He also thanks B. Fantechi for explanation about details in the paper \cite{BF} during the first 
author's stay in SISSA, Trieste as a postdoctor during 2007-2009, and thanks Y-B. Ruan for his lectures and 
generous discussion introducing to him the Landau Ginzburg theory. The second author is partially support
by NSF grant.

\vsp
\noindent
{\bf Conventions}. 
In this paper, the primary focus is on moduli of stable morphisms with fields to $\Pf$, to a smooth quintic Calabi-Yau
$Q\sub \Pf$ defined by $\sum x_i^5=0$, and a deformation of $\Pf$ to the normal cone to $Q\sub\Pf$. 
%Also, we need to work with these moduli spaces coupled with $p$-fields. 
%The managing of comprehensible notation is a challenge. 

Throughout the paper, %, we will abide with the following convention.
we fix a homogeneous coordinates $[x_1,\ldots,x_5]$ of $\Pf$, with $x_i\in H^0(\Pf, \sO(1))$ and $\sO(1)\defeq\sO_{\Pf}(1)$.
We %will denote by $Q\sub \Pf$ the Fermat quintic $\sum x_i^5=0$ in $\Pf$, and 
denote by $N$ the normal bundle to $Q$ in $\Pf$.
Using the defining section $\sum x_i^5=0$, we obtain a canonical isomorphism $N\cong \sO_Q(5)$.

%To prove the equivalence of invariants, we will use a deformation of $\Pf$ to the normal bundle $N$: 
%we define $V=\text{bl}_{Q\times 0} \Pf\times\Ao-R$,
%where $R$ is the proper transform of $\Pf\times 0$ in the blowing up of 
%$\Pf\times\Ao$ along $Q\times 0$.

In this paper, we will fix positive integers $g$ and $d$ throughout. 
%We will use {\sl mathcal} fonts of $Q$, etc.,
%with subscript
%``$g$'' to denote the moduli of genus $d$ degree $d$ stable morphisms to $Q$, as shown below:
%$$\cP_g=\bcM_g(\Pf,d), \ \cQ_g=\bcM_g(Q,d),\  \cN_g=\bcM_g(N,d),\ \cV_g=\bcM_g(V,d).
%$$
%We will add superscript $p$ to denote these moduli spaces coupled with $p$-fields:
%$$\cP_g^p=\bcM_g(\Pf,d)^p,\  \cQ_g^p=\bcM_g(Q,d)^p,\  \cN_g^p=\bcM_g(N,d)^p,\  \cV_g^p=\bcM_g(V,d)^p.
%$$
We will use $(f,\cC)$ with subscripts to denote the universal families of various moduli spaces. 
For instance, after abbreviating $\cP=\bcM_g(\Pf,d)^p$, the universal curve and map of
$\cP$ is denoted by 
$$(f_{\cpg}, \pi_{\cpg}): \cC_{\cpd}\lra \Pf\times\cpd.
$$

For any locally free sheaf $\sL$ on $\cC$, we denote by $\Vb(\sL)$ the underlying
vector bundle of $\sL$; namely, the sheaf of sections of $\Vb(\sL)$ is $\sL$.

In this paper, we will use fonts $\EE$, etc. to denote derived objects (of complexes). We reserve $\LL_{X/Y}$ to denote the
cotangent complex of $X\to Y$; we denote by $\TT_{X/Y}$ its derived dual $\TT_{X/Y}=\LL_{X/Y}\dual$, called the tangent
complex of $X\to Y$. We use $\phi_{X/Y}:\TT_{X/Y}\to\EE_{X/Y}$ to denote a relative obstruction theory of $X\to Y$,
following Behrend-Fantechi \cite{BF}.

Without causing confusion, all pull back of derived objects (resp. sheaves) are derived pull back (resp. sheaves pull back)
unless otherwise stated. 

%{\red Note the exception that for the projection $\pi$, we purposely omitted the subscript, since adding in
%is too cumbersome. By keeping $\pi$ solely or the projection from the universal curves to the base stack, we hope
%that the context will make the meaning of $\pi$ apparent. For instance, $\pi\lsta f_{\cpd}\sta\sO(1)$
%is the direct image sheaf under $\pi: \cC_{\cpd}\to \cpd$ of the full back of $\sO(1)\defeq \sO_{\Pf}(1)$.}

\section{Direct image cones and moduli of sections}
\def\Sbul{\mathrm{Sym}}
\def\be{{\mathbf e}}

In this section, to a locally free sheaf $\cL$ over a family of nodal curves $\pi: \cC\to\fA$ over an Artin stack $\fA$,
we will construct its direct image cone $C(\pi\lsta \cL)$, % that set theoretically is $\{(z, P)\mid z\in\fA, P\in H^0(\cC_z, \cL_z)\}$.
and its relative obstruction theory. %, and exhibit examples of such cones we will use later.

\subsection{Direct image cones}
Let $\fA$ be an Artin stack, $\pi:\cC\to\fA$ be a flat family of connected, nodal, arithmetic genus $g$ curves,
and $\sL$ a locally free sheaf on $\cC$.

\begin{defi}\label{def-1}
For any scheme $S$, we define $C(\pi\lsta \sL)(S)$ be the collection of $(\rho,p)$ so that
$\rho: S\to\fA$ is a morphism and $p\in H^0(\cC_S, \rho\sta \sL)$, where $\cC_S=S\times_\fA \cC$
and $\rho\sta\sL=\sL\times_{\sO_\cC}\sO_{\cC_S}$.

An arrow from $(\rho, p)$ to $(\rho',p')$ in $C(\pi\lsta \sL)(S)$ consists of an arrow $\tau: \cC_S\to \cC_S$ in 
$\fA(S)$ such that %for $\tilde \tau:\cC_S\to\cC_S$ the isomorphism induced by $\tau$, and 
under the induced ismorphism $ \tau\sta\rho^{\prime\ast} \sL\cong \rho\sta\sL$,
$p=\tau\sta p'$.
Given $S\to S'$, we define $C(\pi\lsta \sL)(S')\to C(\pi\lsta \sL)(S)$ by pull back.
\end{defi}

We show that $C(\pi\lsta \sL)$ is a stack over $\fA$. Given a module $\cF$, we denote by
$\Sbul \cF$ the algebra of symmetric product of $\cF$.

\begin{prop}\label{stack}
Let the notation be as in Definition \ref{def-1}. We have canonical $\fA$-isomorphism
$$C(\pi\lsta\sL)\cong \spec_\fA \Sbul R^1 \pi\lsta (\sL\dual\otimes\omega_{\cC/\fA}).
$$
\end{prop}

\begin{proof}
%We denote 
%$$\text{Cone}(\pi\lsta \sL)=\spec_\fA \text{Symm}_{\sO_{\fA}}^\bullet R^1 \pi\lsta (\sL\dual\otimes\omega_{\cC/\fA}).
%$$
For any scheme $S$ and a morphism $\rho: S\to \fA$, we let
$$C(\pi\lsta \sL)(\rho)=\{(\rho, p)\mid p\in H^0(\cC_S,\rho\sta\sL)\}\cong \Gamma(\cC_S, \rho\sta\sL).
$$
We define a transformation
\beq\label{tran-1}
C(\pi\lsta \sL)(\rho)\lra
\Hom_S\bl S, \spec_\fA \Sbul R^1 \pi\lsta (\sL\dual\otimes\omega_{\cC/\fA})\times_\fA S) %\text{Cone}(\pi\lsta \sL)\times_{\fA} S\br
\eeq
as follows.
We let $\sF=R^1\pi\lsta (\sL\dual\otimes \omega_{\cC/\fA})$.  Given a $\rho: S\to\fA$,
an $S$-morphism $S\to \Spec_\fA \Sbul  \sF\times_\fA S$ is given by 
a morphism of sheaves of $\sO_\fA$-algebra 
$$\Sbul \sF  \lra \sO_S,
$$
which is equivalent to a morphism of sheaves of $\sO_\fA$-modules
%$\sF \lra \sO_S$, and equivalently
$$R^1\pi_{S \ast} (\sL_S\dual \otimes \omega_{\cC_S/S})=\sF\otimes_{\sO_\fA}\sO_S\lra \sO_S.
$$
Here we have used the base change property of $R^1\pi\lsta$. %, $\sF\otimes_{\sO_\fA}\sO_S=R^1\pi_{S \ast} (\sL_S\dual \otimes \omega_{\cC_S/S})$.

Applying Serre duality \cite{Conrad} 
to the complete intersection morphism $\pi_S:\cC_S\to S$,
we obtain
$$ \Hom_S(R^1\pi_{S \ast} (\sL_S\dual \otimes \omega_{\cC_S/S}),\sO_S)= \Gamma({\cC_S}, \sL_S).
$$
This defines the transformation \eqref{tran-1}. It is direct to check that this is an isomorphism, and satisfies base change
property. This proves the Proposition.
%Hence the two categories admit a canonical one-one correspondence.
\end{proof}

\subsection{Moduli of sections}\label{sec2.2}

One can also construct the direct image cone via the moduli of sections. 
Let $\cC\to\fA$ be as in Definition \ref{def-1}; let $\cZ\to\cC$ be an Artin stack such that the arrow $\cZ\to \cC$ is
representable and quasi-projective. We define a groupoid $\fS$ (with dependence on $\cZ$ implicitly understood)
as follows.

For any scheme $S\to\fA$, we denote
$\cC_S=\cC\times_\fA S$ and $\cZ_S=\cZ\times_\cC \cC_S$; we view $\cZ_S$ as a scheme over $\cC_S$
via the projection $\pi_S:\cZ_S\to \cC_S$.
We define 
$$\fS{}(S)=\{ s:\cC_S\to \cZ_S\mid s \text{ are $\cC_S$-morphisms}\,\}.
$$
The arrows are defined by pull backs.

\begin{prop}\label{ZA}
The groupoid $\fS{}$ is an Artin stack with a natural projection to $\fA$. 
The morphism $\fS{}\to \fA$ is representable and quasi-projective.
\end{prop}

\begin{proof}
This follows from the functorial construction of Hilbert scheme and that $\cZ\to\cC$ is representable and
quasi-projective.
\end{proof}

\begin{coro}
Let $\pi:\cC\to\fA$ be as in Definition \ref{def-1}, and let $\cZ=\Vb(\sL)$, which is the underlying vector 
bundle of the locally free sheaf $\sL$. Then canonically $C(\pi\lsta \sL)\cong \fS{}$
as stacks over $\fA$.
\end{coro}

\subsection{The obstruction theory}\label{sec2.3}

We give the perfect obstruction theory of $\fS{}$. 
Let $\cZ\to\cC\to\fA$ be as in Proposition \ref{ZA}. Let $\pi_\fS:\cC_\fS \to \fS$
%$(f_{\fS(\cC)}, \pi_{\fS(\cC)}): \cC_{\fS{}}\to \cZ\times \fA$
be the universal family of $\fS{}$ and let $\ee:\cC_\fS\to \cZ$ be the tautological
evaluation map. Namely, $(\pi_\fS,\ee):\cC_\fS\to \fS\times\cZ$ is the universal family
of $\fS$.

As mentioned at the end of the introduction, we let $\TT_{\fS/\fA}$ be the tangent complex of $\fS\to \fA$,
which is the dual of the cotangent complex $\LL_{\fS/\fA}$.

\begin{prop}\label{deformation}\label{DEF}
Let the situation be as stated. Suppose $\cZ\to\cC$ is smooth, then $\fS{}\to\fA$ has a
perfect relative obstruction theory
$$\phi_{\fS{}/\fA}: \TT_{\fS{}/\fA}\lra \EE_{\fS{}/\fA}\defeq R^\bullet \pi_{\fS\ast} \ee\sta \Omega_{\cZ/\cC}\dual.
$$
\end{prop}

\begin{proof}
%For convenience we abbreviate $\fS=\fS{}$.
By our construction, we have the commutative diagrams
\beq\label{first-Q}
\begin{CD}
\fS@<<<\cC_{\fS}@>{\ee}>>\cZ\\
@VVV@VVV@VVV\\
\fA@<<<\cC@>{=}>>\cC,
\end{CD}
\eeq
where the left one is Cartesian. 
Applying the projection formula to
\beq\label{bfQ}
\pi_{\fS}^\ast \TT_{\fS/\fA}\cong \TT_{\cC_\fS/\cC}\lra 
 \ee^\ast \TT_{\cZ/\cC}=
\ee^\ast \Omega\dual_{\cZ/\cC},
\eeq
and using
$$\TT_{\fS/\fA}\lra R^\bullet\pi_{\fS\ast}\pi_\fS\sta\TT_{\fS/\fA},
$$
we obtain
\begin{equation}\label{def--Q}
\phi_{\fS/\fA}: % Hom_{\cC_\fS(\pi_\fS^\ast\TT_{\fS/\fA}, \ee^\ast T_{\cZ/\cC})=Hom_{\fS} (
\TT_{\fS/\fA}\lra \EE_{\fS/\fA}\defeq R^\bullet \pi_{\fS\ast} \ee^\ast \Omega\dual_{\cZ/\cC}.
\end{equation}
We claim that $\phi_{\fS/\fA}$ is a perfect obstruction theory.
%$$\phi_{\fS/\fA} : \TT_{\fS/\fA}\lra R^\ast\pi_{\fS\ast} f_\fS^\ast  \sK :=\EE^\bullet_{\fS/\fA}.
%$$

We prove this by applying the criterion in \cite[Thm 4.5]{BF}. 
Given an extension $T\sub T'$ by ideal $J$ with $J^2=0$, and a commutative diagram
% a morphism $\mm: T\to\fS$ and an $\underline{\mm} : T'\to\fA$ fitting into the commutative square
\beq\label{lift-Q}
\begin{CD}
T@>{\mm}>>\fS\\
@VVV@VVV\\
T'@>{\nn}>>\fA,
\end{CD}
\eeq
we say that $\mm$ lifts to an $\mm':T' \to \fS$ if $\mm'$ fits into \eqref{lift-Q} to form two commuting triangles.

By standard deformation theory,
the diagram (\ref{lift-Q}) provides a morphism
$$\mm^\ast\LL_{\fS/\fA}\lra \LL_{T/T'}\lra \LL^{\geq -1}_{T/T'}=J[1],
$$
which gives an element
\beq\label{obclass}
\varpi(\mm)\in   \Ext^1_T(\mm^\ast\LL_{\fS/\fA},J)= H^1(T,\mm^\ast\TT_{\fS/\fA}\otimes_{\sO_T} J).
\eeq
Using the morphism $\phi_{\fS/\fA}$ in \eqref{def--Q}, we obtain the homomorphism
$$\phi':H^1(T,\mm^\ast\TT_{\fS/\fA}\otimes_{\sO_T} J)\lra H^1(T,\mm^\ast \EE_{\fS/\fA} \otimes_{\sO_T} J).
$$
We define
$$
\ob(T, T', \mm):= \phi'(\varpi(\mm))\in H^1(T,\mm^\ast \EE_{\fS/\fA} \otimes_{\sO_T} J).
$$

%By repeating the argument in Lemma \ref{def-sec}, we obtain the element
%$$\varpi(\mm)\in    H^1(T,\mm^\ast\TT_{\fS/\fA}\otimes_{\sO_T} J),
%$$
%the homomorphism
%$$\phi'=H^1(\mm^\ast\phi_{\fS/\fA}\otimes J):H^1(T,\mm^\ast\TT_{\fS/\fA}\otimes_{\sO_T} J)\lra H^1(T,\mm^\ast \EE_{\fS/\fA} \otimes_{\sO_T} J),
%$$
%and the image
%$$
%\ob(T, T', \mm):= \phi'(\varpi(\mm))\in H^1(T,\mm^\ast \EE_{\fS/\fA} \otimes_{\sO_T} J).
%$$

To prove that $\phi_{\fS/\fA}$ is a perfect relative obstruction theory, 
by the criterion in \cite[Thm 4.5 (3)]{BF}, we need to show 
\begin{enumerate}
\item
$\ob(T, T', \mm)=0$ if and only if $\mm$ in (\ref{lift-Q})
can be lifted to $\mm':T'\to \fC$;
\item when $\ob(T,T',\mm)=0$, the set of liftings $\mm': T'\to \fC$ form a torsor under
$H^0(T,\mm^\ast\EE_{\fC/\fA}\otimes_{\sO_T} J)$.
 \end{enumerate}
 
 We now verify (1) and (2). 
Pulling back $\cC$ to $T$ and $T'$ via $\mm$ and $\nn$, we obtain two families
$\pi_T: \cC_T\to T$ and $\pi_{T'}:\cC_{T'}\to T'$; pulling back
$\ee$ to $T$, we have evaluation map $\ee_T:\cC_T\to \cZ$.
Let  
$$\kappa:H^1(T,\mm^\ast\EE_{\fS/\fA}\otimes_{\sO_T} J)
%H^1(T,m^\ast q^\ast R^\bullet\tau\lsta \sL_\fA \otimes_{\sO_T} J)
\mapright{\cong} H^1(T,R^\bullet\pi_{T\ast}(\ee_T^\ast \Omega_{\cZ/\cC}\dual\otimes \pi_T\sta J))
$$
be the canonical isomorphism defined by the definition of $\EE_{\fS/\fA}$ (cf. \eqref{def--Q}).
%($H^1(T,R^\bullet \pi_{T\ast}\sL_T\otimes J)$)

Using the standard property of cotangent complex, the commuting square
\beq\label{lift-CQ}
\begin{CD}
\cC_T@>{\ee_T}>> \cZ\\
@VVV @VVV\\
\cC_{T'}@>{\tilde\nn}>> \cC,
\end{CD}
\eeq
where $\ti\nn$ is the lift of $\nn$ in \eqref{lift-Q}, induces homomorphisms
$$\ee_T^\ast \Omega_{\cZ/\cC}  \cong\ee_T^\ast\LL_{\cZ/\cC}\lra \LL_{\cC_T/\cC_{T'}} = \pi_T^\ast\LL_{T/T'}
\lra \LL_{\cC_T/\cC_{T'}}^{\geq -1}=  \pi_T^\ast J[1].
$$
Their composite associates to an element
$$\varpi(\be_T,\cZ,\cC)\in H^1(\cC_T,\ee_T^\ast \Omega\dual_{\cZ/\cC}  \otimes \pi_T^\ast J)\cong H^1(T,R^\bullet\pi_{T\ast}(\ee_T^\ast \Omega\dual_{\cZ/\cC}\otimes\pi_T^\ast J)).
$$ 
By Lemma \ref{ob-general}, $\varpi(\be_T,\cZ,\cC)=0$ if and only if (\ref{lift-CQ}) admits a lifting $\cC_{T'}\to \cZ$.

As (\ref{lift-CQ}) is the composition of (\ref{lift-Q}) with (\ref{first-Q}), $\varpi(\be_T,\cZ,\cC)=\kappa(\phi'(\varpi(\mm)))$. 
Thus $\ob(T,T',m)=0$ if and only if (\ref{lift-CQ}) has a lifting, which is equivalent to that $\mm$ lifts to an 
$\mm':T' \to \fS$ in (\ref{lift-Q}). This verifies criterion (1).

Finally, when $\ob(T,T',\mm)=0$, any 
two liftings $\cC_{T'}\to \cZ$ differ by a section in $H^0(\cC_T,\ee_T^\ast \Omega\dual_{\cZ/\cC}  \otimes \pi_T^\ast J)$,  and vice versa 
\cite[Thm 2.1.7]{Illusie}. This proves the criterion (2).
%\black This proves that the set of liftings form a torsor under
%$H^0(\cC_T,\sK_T\otimes \pi_T^\ast J)=  H^0(T,\mm^\ast\EE^\bullet_{\cqg/\fA}\otimes J)$.
These complete the proof of the Proposition. 
\end{proof}

\subsection{Moduli of stable morphisms}\label{sec-2.2}

Using the stack $\fdg$ of curves with line bundles, this construction provides a different perspective of
the moduli of stable morphisms to a projective scheme. 

\begin{defi}
We define $\fdg$ be the groupoid associating to each scheme $S$ the set $\fdg(S)$ of pairs $(\cC_S,\sL_S)$, where
$\cC_S\to S$ is a flat family of connected nodal curves and $\sL_S$ is a line bundle on $\cC_S$ of degree $d$
along fibers of $\cC_S/S$. An arrow from $(\cC_S,\sL_S)$ to $(\cC_S',\sL_S')$ consists of a pair $(\rho,\tau)$, where
$\rho: \cC_S\to\cC_S'$ and $\tau: \rho\sta\sL_S'\to \sL$ are $S$ isomorphisms.
\end{defi}

It is easy to show that $\fdg$ is a smooth Artin stack.
By forgetting the line bundles one obtains an induced morphism
$\fdg\to \fM_g$, where $\fM_g$ is the Artin stack of all connected genus $g$ nodal curves.
For any $\xi=(C,L)\in\fdg$, the automorphism group of $\xi$ relative to $\fM_g$,
(i.e. automorphisms of $L$ that fix $C$,) is $\CC^\ast$. We denote by $(\cC_{\fdg},\cL_{\fdg})$, with
$\pi_{\fdg}: \cC_{\fdg}\to \fdg$ implicitly understood, the universal family of
$\fdg$.

\vsp

We now let $X\sub \Pn$ be a projective scheme.
For the integer $d$ given, (the integer $d$ will be fixed throughout this paper,) we have the moduli of
genus $g$ and degree $d$ stable morphisms to $X$: $\bcM_g(X,d)$. 
We now present it as a moduli of sections. We keep the homogeneous coordinates
$[x_1,\ldots, x_{n+1}]$ of $\Pn$ mentioned in the introduction. The choice of $[x_i]$
provides a presentation 
\beq\label{pre-P}
\Pn=\AA^{n+1 \ast}/\CC\sta,\quad \AA^{n+1 \ast}:=\AA^{n+1}-0.
\eeq
%We take the stack $\fdg$ of pairs of curves with degree $d$ line bundles. Let $\pi_{\fdg}: \cC_{\fdg}\to\fdg$
%be the universal family and $\sL_{\fdg}$ the universal invertible sheaf on $\cC_{\fdg}$. 
We form the bundle 
$$\Vb(\sL_\fdg^{\oplus (n+1)})\sta=\Vb(\sL_\fdg^{\oplus (n+1)})-0_{\cC_\fdg},
$$
where $0_{\cC_\fdg}$ is the zero section. Using the %presentation $\Pn=(\CC^{n+1}-0)/\CC\sta$, we have
$\CC\sta$-equivariance of the projection $\AA^{n+1 \ast}\to \Pn$ induced by \eqref{pre-P}, we obtain a canonical morphism
$$\Psi: \Vb(\sL_\fdg^{\oplus (n+1)})\sta\lra \Pn.
$$
We let 
$$\cZ_X=\Vb(\sL_\fdg^{\oplus (n+1)})\sta\times_{\Pn} X\sub \Vb(\sL_\fdg^{\oplus (n+1)})\sta.
$$
We let $\fS_X$ be the stack of sections constructed in Subsection \ref{sec2.2} with $\cZ$ replaced by $\cZ_X$. 

\begin{prop}
%The stack $\fS(\cZ)$ is a Deligne-Mumford stack.
There is a canonical open immersion of stacks $\bcM_g(X,d)\to \fS_X$, as stacks over $\fM_g$.
\end{prop}

\begin{proof}
For notational simplicity, in the remainder of this Section, we abbreviate $Y=\bcM_g(X,d)$,
and denote by
$(f_Y,\pi_Y):\cC_Y\to X\times Y$ the universal family. Pulling back
$\sO(1)$, we obtain $\sL_Y=f_Y\sta\sO(1)$; pulling back the homogeneous coordinates $x_i$, (viewing
$x_i\in H^0(\Pn,\sO_{\Pn}(1))$), we obtain $u_i=f_Y\sta x_i$. Since $f_Y$ has degree $d$ along fibers of $\cC_Y/Y$,
$(\cC_Y,\sL_Y)$ defines a morphism 
\beq\label{lam-1}
\lam: Y=\bcM_g(X,d) \lra \fdg;
\eeq
since $f_Y(\cC_Y)\sub X$, $(u_1,\ldots, u_{n+1})$ defines a section
$Y\to \Vb(\sL_\fdg^{\oplus (n+1)})\times_{\fdg} Y$
(of $Y$)
that factors through a section
$$\xi:  Y\to\cZ_X \times_{\fdg}Y.
$$
This defines a morphism $Y\to \fS_X$.

It is direct to check that this is an open immersion, and is a morphism over $\fM_g$.
This proves the Proposition.
\end{proof}

It is worth comparing the relative obstruction theory $\phi_{Y/\fD_g}$ of $\bcM_g(X,d)\to \fdg$
constructed using Subsection \ref{sec2.3} with the relative obstruction theory $\phi_{Y/\fM_g}$ of 
$\bcM_g(X,d)\to \fM_g$ 
given in \cite{BF}.

 %using Proposition \ref{DEF}.
%%Let $Z\defeq \AA^{\! n+1}-0\to \Pn$ be the projection given by the homogeneous coordinates $[x_i]$;
%Let 
%\beq\label{W}W_X=W\times_{\Pn}X,\quad \gamma: W_X\to X
%\eeq
%be the projection.
%We have the commutative diagram of exact sequences 
% $0\lra \sO_Z\to\Omega_Z\dual\to \gamma\sta\Omega_{\Pn}\dual\to 0$ is of the form
%$$
%\begin{CD}
%0@>>> \sO_{W_X}@>>> \gamma\sta \sO_{\Pn}(1)^{\oplus (n+1)} (=\gamma\sta\Omega_W\dual) 
%@>>> \gamma\sta \Omega_{\Pn}\dual @>>> 0\\
%@. @| @AAA @AAA\\
%0@>>> \sO_{W_X}@>>> \Omega_{W_X}\dual @>>> \gamma\sta \Omega_X\dual @>>> 0\\
%\end{CD}
%$$

%Let $\lam$ be the projection $\lam:\bcM_g(X,d)\to \fdg$.

Following the notation before Proposition \ref{deformation}, we have an evaluation map 
$\ee_Y:\cC_Y\lra \cZ_X$. 
The induced morphism 
$\pi_Y^\ast\TT_{Y/\fdg}\cong\TT_{\cC_Y/\cC_\fdg}\to \ee_Y^\ast\TT_{\cZ_X/\cC_\fdg}$
induces
$$\phi_{Y/\fdg}:\TT_{Y/\fdg}\lra \EE_{Y/\fdg}:=R^\ast\pi\lsta\ee_Y^\ast\TT_{\cZ_X/\cC_\fdg}.$$
Applying Proposition \ref{DEF}, $\phi_{Y/\fdg}$ is a perfect relative obstruction theory 
of $Y\to \fdg$.

\begin{lemm}\label{M-D}
Suppose $X$ is smooth.  The relative obstruction theories
$\phi_{Y/\fdg}$ 
and $\phi_{Y/\fM_g}$ are related by a morphism of distinguished triangles 
$$
\begin{CD}
R^\ast\pi_{Y\ast} \sO_{\cC_Y}@>>>\EE_{{Y}/\fdg}@>>>\EE_{{Y}/\fM_g}@>{+1}>>\\
@AA{||}A@AA{\phi_{{Y}/\fM_g}}A@AA{\phi_{{Y}/\fdg}}A\\
\lam^\ast\TT_{\fdg/\fM_g}[-1]@>>>\TT_{{Y}/\fdg}@>>>\TT_{{Y}/\fM_g}@>{+1}>>.\\
\end{CD}
$$
\end{lemm}

\begin{proof}
Let $\cC_{\fM_g}$ be the universal curve on $\fM_g$; let
$$\chi_M:\cZ_X\lra \cC_{\fM_g}\times X
$$ 
be the morphism so that its first factor is the composite $\cZ_X\to\cC_\fdg\to\cC_{\fM_g}$,
and the second factor is the natural projection. Let
$$f:\cC_Y\lra \cC_{\fM_g}\times X
$$
be the composte of $\ee_Y:\cC_Y\to\cZ_X$ with $\chi_M:\cZ_X\to\cC_{\fM_g}\times X$. Note that the first factor 
of $f$ is the canonical projection induced by $Y\to \fdg\to\fM_g$; its the second factor is $f_Y$. 

Taking the tangent complex relative to $\fM_g$, we obtain
$$\pi_Y^\ast\TT_{Y/\fM_g}\cong\TT_{\cC_Y/\cC_{\fM_g}}\lra f^\ast \TT_{\cC_{\fM_g}\times X/\cC_{\fM_g}}\cong
 f_Y^\ast T_{X}.
$$
This induces 
$$\phi_{Y/\fM_g}:\TT_{Y/\fM_g}\lra \EE_{Y/\fM_g}:=R^\ast\pi\lsta f_Y^\ast T_X,
$$
which is the  perfect relative obstruction theory of $Y\to \fM_g$ defined in \cite{BF}. 

We let $\chi_D:\cZ_X\to\cC_\fdg\times X$ be defined similar to $\chi_M$,  and let $g:\cC_\fdg\times X\to\cC_{\fM_g}\times X$ be the projection.
Note that $g\circ \chi_D=\chi_M$.
By the construction, we have the commutative diagrams
\beq\label{bigdia}
\begin{CD}
\cZ_X@>{\chi_D}>>\cC_\fdg\times X@>{g}>>\cC_{\fM_g}\times X\\
@VV{\rho_0}V@VV{\pi_1}V@VVV\\
\cC_\fdg@=\cC_\fdg@>>>\cC_{\fM_g}.
\end{CD}
\eeq

It induces an exact sequence of locally free sheaves
$$0\lra  T_{\cZ_X/\cC_\fdg\times X}
\lra T_{\cZ_X/\cC_\fdg}\lra \chi_D^\ast T_{\cC_\fdg\times X/\cC_\fdg}%\cong\chi_M^\ast T_{\cC_{\fM_g}\times X/\cC_{\fM_g}}
\lra 0.
$$
Since $\chi_D$ is a $\CC\sta$-principal bundle, $\sO_{\cZ_X}\cong T_{\cZ_X/\cC_\fdg\times X}$.
Also we have canonical isomorphism $\chi_D^\ast T_{\cC_\fdg\times X/\cC_\fdg}\cong\chi_M^\ast T_{\cC_{\fM_g}\times X/\cC_{\fM_g}}$.
Let $\lam_C:\cC_Y\to\cC_\fdg$ be induced by $\lam$. The above sequence fits into a
morphism of distinguished triangles
$$%\beq\label{commute}
\begin{CD}
\ee_Y^\ast T_{\cZ_X/\cC_\fdg\times X}@>>> \ee_Y^\ast T_{\cZ_X/\cC_\fdg}@>>>\ee_Y^\ast\chi_M^\ast T_{\cC_{\fM_g}\times X/\cC_{\fM_g}}\cong f_X^\ast T_X@>{+1}>>\\
@AAA@AAA@AAA\\
\lam_C^\ast\TT_{\cC_\fdg/\cC_{\fM_g}}[-1]@>>> \TT_{\cC_Y/\cC_\fdg}@>>>\TT_{\cC_Y/\cC_{\fM_g}}@>{+1}>>
\end{CD}
$$%\eeq
where the left vertical arrow is the composition
$$\lam^\ast\TT_{\cC_\fdg/\cC_{\fM_g}}\cong \ee_Y^\ast\chi_D^\ast \pi_1^\ast T_{\cC_\fdg/\cC_{\fM_g}}\cong \ee_Y^\ast\chi_D^\ast T_{\cC_\fdg\times X/\cC_{\fM_g}\times X}\lra \ee_Y^\ast T_{\cZ_X/\cC_\fdg\times X}[1],
$$
where the last arrow is given by the distinguished triangle of contangent complexes associated to the top row of (\ref{bigdia}).
Here the commutativity of squares in the above diagram can be checked by diagram chasing using \eqref{bigdia}).

Therefore we have a homomorphism of distinguished triangles
$$ 
\begin{CD}
R^\ast\pi_{Y\ast} \sO_{\cC_Y}@>>>\EE_{Y/\fdg}@>>>\EE_{Y/\fM_g}@>{+1}>>\\
@AAA@AA{\phi_{Y/\fdg}}A@AA{\phi_{Y/\fM_g}}A\\
\lambda^\ast\TT_{\fdg/\fM_g}[-1]@>>>\TT_{Y/\fdg}@>>>\TT_{Y/\fM_g}@>{+1}>>.
\end{CD}
$$
By the property of contangent complex of Picard stacks the left vertical arrow
of the above diagram is an isomorphism.
\end{proof}

Let $[{Y}/\fdg]\virt$ and $[{Y}/\fM_g]\virt\in A\lsta {Y}$ be the virtual cycles
using the respective perfect relative obstruction theories. 

\begin{coro}
We have identity 
$$[{Y}/\fdg]\virt=[{Y}/\fM_g]\virt\in A\lsta {Y}.
$$
\end{coro}

\begin{proof}
 Applying \cite[prop 2.7]{BF} to Lemma \ref{M-D}, we obtain a diagram of cone stacks
$$ 
\begin{CD}
h^1/h^0(R^\ast\pi_{Y\ast}\sO_{\cC_Y}) @>>> h^1/h^0(\EE_{Y/\fdg})@>{\theta}>> h^1/h^0(\EE_{Y/\fM_g})\\
@|@AA{(\phi_{Y/\fdg})\lsta}A@AA{(\phi_{Y/\fM_g})\lsta}A\\
h^1/h^0(\lambda^\ast\TT_{\fdg/\fM_g}[-1])@>>> h^1/h^0(\TT_{Y/\fdg})@>{\theta_{\text{int}}}>> h^1/h^0(\TT_{Y/\fM_g})
\end{CD}
$$
of which the two rows are exact sequence of abelian cone stacks.
Applying argument analogous to the second line in the proof of \cite[Prop 3]{Kresch2}, one checks $(\theta_{\text{int}})\sta(C_{Y/\fM_g})=C_{\fdg/\fM_f}$.
Hence $\theta$ is a quotient of bundle stacks such that $\theta^\ast (C_{Y/\fM_g})=C_{Y/\fdg}$. By projection formula
$$ [Y/\fdg]\virt= [Y/\fM_g]\virt\in A\lsta Y.
$$
This proves the Corollary.
\end{proof}

\section{Gromov-Witten invariant of the GSW model }
\def\upf{^{\oplus 5}}
\def\umtf{{-\otimes 5}}

In this section, we will construct the moduli of stable morphisms to $\Pf$ coupled with $p$-fields.
We will construct its localized virtual cycle, using Kiem-Li's cosection localized virtual cycles. We define its
degree be the virtual counting of stable maps to $\Pf$ with $p$-field. This class of invariants is a generalization
of genus zero Guffin-Sharpe-Witten model $(K_\Pf, \bw_\Pf)$  \cite{Sharpe}.

\subsection{Moduli of stable maps with $p$-fields}\label{sec3.1}
\def\fu{{\mathfrak u}}

Let $\bcM_g(\Pf,d)$ be the moduli of genus $g$ degree $d$ stable maps to $\Pf$. 
%For simplicity,
%in this paper we abbreviate 
%$$\cpd=\barM_g(\Pf,d).
%$$
For the moment, we denote by
$(f_{M}, \cC_{M},\pi_{M})$ be the universal family of $\bcM_g(\Pf,d)$, and $\sL_{M}=f_{M}\sta\sO(1)$
the tautological line bundle.  We form 
$$\sP_{M}\defeq \sL_{M}^{-\otimes 5}\otimes \omega_{\cC_{M}/M},
$$
and call it the auxiliary invertible sheaf on $\bcM_g(\Pf,d)$.

We define the moduli of genus $g$ degree $d$ stable morphisms with $p$-fields be the direct image cone:
\beq\label{M-P}
\cpgp\defeq \bcM_g(\Pf,d)^p\defeq C(\pi_{M\ast} \sP_{M}). %\quad \sP_{M}\defeq \sL_{\cpd}^{-\otimes 5}\otimes \omega_{\cC_{\cpd}/\cpd}.
\eeq
(We abbreviate it to $\cpgp$, as indicated above.)

\vsp
Like before, we can embed $\cpgp$ into the moduli of sections for a choice of $\cZ\to\fdg$.
Let $[x_1,\ldots,x_5]$ be the homogeneous coordinates of $\Pf$ specified in the Introduction.
Let 
$$(f_{\cpgp},  \pi_{\cpg}):  \cC_{\cpgp}\to \Pf\times\cpgp
$$ 
be the universal map
of $\cpgp$. We let $\sL_{\cpgp}=f_{\cpg}\sta\sO(1)$ the tautological invertible sheaf; let
$\sP_{\cpgp}=\sL_{\cpgp}^{-\otimes 5}\otimes\omega_{\cC_{\cpgp}/\cpgp}$ be the auxiliary invertible sheaf,
and let 
\beq\label{pu}
\fp\in \Gamma(\cC_{\cpgp},\sP_{\cpgp})\and
\fu_i=f_{\cpg}\sta x_i\in \Gamma(\cC_{\cpg}, \sL_{\cpg})
\eeq
be the universal $p$-field and the tautological coordinate functions, respectively.
Note that $(\cC_{\cpgp}, \sL_{\cpgp})$ induces a morphism $\cpgp\to\fdg$ so that $(\cC_\cP,\sL_\cP)$ is isomorphic
to the pull back of $(\cC_\fdg,\sL_\fdg)$.

Using the line bundle $\sL_{\fdg}$ on $\cC_{\fdg}$ and its auxiliary invertible sheaf
$$\sP_\fdg=\sL_{\fdg}^{-\otimes 5}\otimes\omega_{\cC_{\fdg}/\fdg},
$$
we form the bundle
$$\cZ\defeq \Vb(\sL_{\fdg}\upf\oplus \sP_\fdg)  %-0_{\fdg})\times_{\cC_\fdg}\Vb(\sP_\fdg)
$$
over $\cC_{\fdg}$. %, where $0_{\fdg}\sub \Vb(\sL_{\fdg}\upf)$ is the zero section.
Then the section $((\fu_i)_{i=1}^5,\fp)$ defines a section of 
$$\cZ\times_{\cC_\fdg}\cC_{\cpgp}\lra \cC_{\cpgp}.
$$
This section induces a $\cC_\fdg$-morphism % to the moduli of sections $\fS(\cZ)$:
$\cC_{\cpgp}\to \cZ\times_{\cC_\fdg}\cC_\cpgp$. Composed with the projection $ \cZ\times_\fdg\cpgp\to\cZ$, we obtain the 
evaluation morphism over  $\cC_\fdg$:
\beq\label{ev-p}
\ti\ee: \cC_\cpgp\lra \cZ.
\eeq

\begin{prop}\label{obW}
%The morphism $\ti\ee$ is an open immersion.
The pair
$\cpdp\to \fdg$ admits a perfect relative obstruction theory
$$\phi_{\cpdp/\fdg}:\TT_{\cpdp/\fdg}\lra \EE_{\cpdp/\fdg}\defeq  
R^\ast\pi_{\cpgp \ast} (\sL_{\cpgp}^{\oplus 5}\oplus   \sP_{\cpgp}).
$$ 
\end{prop}

\begin{proof}
 The proof follows from Proposition \ref{DEF} applied to the (evaluation) morphism $\ti\ee$,
using that $\Omega_{\cZ/\cC_\fdg}\dual=\sL_{\fdg}\upf\oplus \sP_\fdg$.
\end{proof}

\subsection{Constructing a cosection}
%We continue to work with the universal family $(\cC_{\fdg},\sL_{\fdg})$ over $\fdg$ and
%the auxiliary sheaf
%$\sP_{\fdg}=\sL_{\fdg}^{-\otimes 5}\otimes \omega_{\cC_{\fdg}/\fdg}$. 
%We continue to use,
%say, $\Vb(\sL_g^{\oplus 5})$ to denote the underlying vector bundle (stack) of the locally free sheaf 
%$\sL_g^{\oplus 5}$.
%$$ \fL_{5\cdot 1}:=\Total (\sL_g^{\oplus 5}) , \quad \fL_P:=\Total (\sL_g^{-\otimes 5}\otimes \omega_{\cC_{\fdg}/\fdg})  , \quad
% \fL_{\omega}:=\Total (\omega_{\cC_{\fdg}/\fdg}).
%$$
%where $\Total(\cdot)$ means the underlying vector bundle of the locally free sheaf, 
%considered as a scheme (stack),
%and denote their respective projections to $\cC_{\fdg}$ by 
%$$q_{5\cdot 1}:\fL_{5\cdot 1}\lra \cC_{\fdg},\quad q_P:\fL_P\lra \cC_{\fdg},\quad q_\omega:\Vb(\omega_{\cC_{\fdg}/\fdg})\lra \cC_{\fdg}.
%$$
We define a multi-linear bundle morphism
 \begin{equation}\label{co}
%\fL_{5\cdot 1}\times_{\cC_{\fdg}} \fL_P 
h_1: \Vb(\sL_{\fdg}\upf \oplus \sP_{\fdg})\lra \Vb(\omega_{\cC_{\fdg}/\fdg}),
%\Vb(\omega_{\cC_{\fdg}/\fdg}),
\quad h_1(z,p)=p\cdot \sum_{i=1}^5 z_i^5,
\end{equation}
where $(z,p)=\bl (z_i)_{i=1}^5,p)\in \Vb(\sL_{\fdg}\upf \oplus \sP_{\fdg})$.  
This map is based on the dual-pairing $\sL_{\fdg}\of\otimes \sP_{\fdg}\to \omega_{\cC_{\fdg}/\fdg}$.
%and $\fp\in \fL_P|_\xi$ 
%over a $\xi\in\cC_{\fdg}$.
%${h_1}(z_1,z_2,z_3,z_4,z_5,p)=p(z_1^5+z_2^5+z_3^5+z_4^5+z_5^5)$. Clearly $\bh$ is a morphism of stacks over $\cC_{\fdg}$.

The morphism $h_1$ induces a homomorphism of tangent complexes
$$%d{h_1}: 
dh_1: \TT_{\Vb(\sL_{\fdg}\upf \oplus \sP_{\fdg})/\cC_{\fdg}}
%\TT_{\Vb(\sL_{\fdg}\upf )/\cC_{\fdg}}\boxplus \TT_{\Vb(\sP_{\fdg})/\cC_{\fdg}}
\mapright{} h_1^\ast 
\TT_{\Vb(\omega_{\cC_{\fdg}/\fdg})/\cC_{\fdg}}=h_1^\ast 
\Omega\dual_{\Vb(\omega_{\cC_{\fdg}/\fdg})/\cC_{\fdg}}.
$$
%$$d{h_1}: \TT_{\Vb(\sL_{\fdg}\upf )\times_{\cC_{\fdg}}\fL_P/\cC_{\fdg}}=
%\TT_{\Vb(\sL_{\fdg}\upf )/\cC_{\fdg}}\boxplus \TT_{\fL_P/\cC_{\fdg}}\lra h_1^\ast \TT_{\Vb(\omega_{\cC_{\fdg}/\fdg})/\cC_{\fdg}},
%$$
%where $\boxplus$ means the direct sum of the pull backs of $\TT_{\Vb(\sL_{\fdg}\upf )/\cC_{\fdg}}$
%and $\TT_{\Vb(\sP_{\fdg})/\cC_{\fdg}}$ %\TT_{\Vb(\sL_{\fdg}\upf )/\cC_{\fdg}}$ and $\TT_{\fL_P/\cC_{\fdg}}$
%via projections $\Vb(\sL_{\fdg}\upf \oplus \sP_{\fdg})\to \Vb(\sL_{\fdg}\upf)$ and $\Vb(\sP_{\fdg})$.
%which is 
%$$d{h_1}: r_1^\ast \sL_g^{\oplus 5}\boxplus r_P^\ast \sP_{\fdg} \lra {h_1}^\ast r_\omega^\ast \omega_{\cC_{\fdg}/\fdg}.$$
In explicit form,  for any closed $\xi\in \cC_{\cpg}$ and $(z,p)\in \Vb(\sL_{\fdg}\upf \oplus \sP_{\fdg})|_\xi$, $dh_1|_{(z,p)}$ sends
$$
((\mathring z_i),\mathring p)\in 
\Omega\dual_{\Vb(\sL_{\fdg}\upf \oplus \sP_{\fdg})/\cC_{\fdg}}\big|_{(z,p)}=(\sL_{\fdg}\upf \oplus \sP_{\fdg})
\otimes_{\sO_{\cC_{\fdg}}}\kk(\xi)
$$
to
\beq\label{dh}
d{h_1}|_{(z,p)}(\mathring z,\mathring p)=(\sum_{i=1}^5 z_i^5)\cdot \mathring p+ p\cdot \sum_{i=1}^5 5z_i^4\cdot \mathring z_i.
\eeq
%Let $\ee$ be the evaluation morphism:
%% From the construction of $\cpdp$ we have an evaluation morphism over $\cC_{\fdg}$
%$$\ee:\cC_{{\cpdp}}\lra \Vb(\sL_{\fdg}\upf \oplus \sP_{\fdg}).
%$$
%($\ee$ is defined by the tautological sections $u_i=f_{\cpd}\sta x_i$ and the universal $p$-field $\fp\in \Gamma(\cC_{\cpg},
%\sP_{\cpg})$.)
 
On the other hand, by pulling back $d{h_1}$ to $\cC_{\cpdp}$ via 
the evaluation morphism $\tilde\ee$ (cf. \eqref{ev-p}) one has
 (homomorphism and canonical isomorphisms)
$$\ti\ee^\ast (dh_1):\ti\ee^\ast \Omega\dual_{\Vb(\sL_{\fdg}\upf \oplus \sP_{\fdg})/\cC_{\fdg}}\lra 
\ti\ee^\ast h_1^\ast\Omega\dual_{\Vb(\omega_{\cC_{\fdg}/\fdg})/\cC_{\fdg}}.
$$
Because the right hand side is canonically isomorphic to 
%\cong \ee^\ast h_1^\ast  q_\omega^\ast \omega_{\cC_{\fdg}/\fdg} \cong 
$\omega_{\cC_{\cpdp}/{\cpdp}}$,
applying $R^\bullet\pi_{{\cpdp}\ast}$, we obtain
\beq\label{map}
\begin{CD}
\sigma_1^\bullet:\EE_{{\cpdp}/\fdg}\lra R^\bullet\pi_{{\cpdp}\ast}(\ee^\ast
h_1^\ast \Omega\dual_{\Vb(\omega_{\cC_{\fdg}/\fdg})/\cC_{\fdg}})\cong R^\bullet\pi_{{\cpdp}\ast}(\omega_{\cC_{\cpdp}/{\cpdp}}).
\end{CD}
\eeq
We define %denote $\sigma_1$ to be $H^1(\sigma_1^\bullet)$:
\beq\label{si-cosection}
\sigma_1:=H^1(\sigma_1^\bullet):\Ob_{{\cpdp}/\fdg}=H^1(\EE_{{\cpdp}/\fdg})\lra %R^1\pi_{{\cpdp}\ast}(\ee^\ast h_1^\ast\TT_{\Vb(\omega_{\cC_{\fdg}/\fdg})/\cC_{\fdg}})\cong 
R^1\pi_{{\cpdp}\ast}(\omega_{\cC_{\cpdp}/{\cpdp}})\cong \sO_{\cpdp}.
\eeq
 
By Proposition \ref{obW}, $\sigma_1$ is in the form (of homomorphism of sheaves)
$$\sigma_1:  \Ob_{\cP/\fdg}=R^1\pi_{\cpdp \ast}\sL_{\cpdp}^{\oplus 5}\oplus R^1\pi_{\cpgp\ast} \sP_{\cpdp}\lra \sO_{\cpdp}.
$$

  %Recall $\fdg^\bs$ is the total stack of $\overline{\pi_{\ast}\sL_g}$ 

%\subsection{at closed point}

% Let $x=(C,L,\phi_i,p)\in {\cpdp}$  be a closed point in the moduli ${\cpdp}$. Let
 %\begin{equation}
 %(\delta\phi_1,\delta\phi_2,\delta\phi_3,\delta\phi_4,\delta\phi_5)\in H^1(C,L^{\oplus 5}) \  \and \delta \fp\in H^1(C,L^{-\otimes 5}\otimes w_C)
 %\end{equation}
 %so  $\delta_x=(\delta\phi_1,\cdots,\delta\phi_5,\delta p)$ is a vector in the respective obstruction space. We associate the  following two numbers
%\begin{equation}\label{sigma1}
 %\sigma_1_1(x,\delta_x)=5\int_C p\cdot (\phi_1^4\cdot \delta\phi_1+\phi_2^4\cdot \delta\phi_2+\phi_3^4\cdot \delta\phi_3+\phi_4^4\cdot \delta\phi_4+\phi_5^4\cdot \delta\phi_5)
% \end{equation}
%and
%\begin{equation}\label{sigma2}
%\sigma_1_2(x,\delta_x)=\int_C (\phi_1^5+\phi_2^5+\phi_3^5+\phi_4^5+\phi_5^5)\cdot \delta p .
 %\end{equation}

\subsection{Degeneracy loci of the cosection}

We give a coordinate expression of the cosection $\sigma_1$.
%For any affine scheme $T$ and a morphism $\eta:T\to {\cpdp}$, the associated morphism $T\to \fdg$ pulls
%back $\sL_g$ and $\cC_{\fdg}$ to form $\pi_T:\cC_T\to T$ and a line bundle $\sL_T$ on $\cC_T$. 
%The morphism $T\to {\cpdp}$ is given by the data
We denoting by
$\fu_i=f_{\cpgp}\sta x_i$ and $\fp\in\Gamma (\cC_{\cpgp},\sP_{\cpgp})$
be the 
%$$\Gamma(\cC_{\cpgp},\sL_{\cpgp}^{\oplus 5})\oplus  
tautological section of $\cpgp$.
%(Note that $[\phi_1,\ldots,\phi_5]$ defines a morphism $f_T: \cC_T\to\Pf$, and $\sL_T=f_T\sta \sO(1)$.)
Take any \'etale chart $T\to\cpgp$, and let $\cC_T=\cC_{\cpgp}\times_{\cpgp} T$.
%\Gamma(\cC_T,\phi_T^\ast\sO(1)^{\oplus 5})\cong \Gamma(\cC_T,\sL_T^{\oplus 5})$$
%and
%$$\fp\in \Gamma(\cC_T,\phi_T^\ast\sO(-5)\otimes w_{\cC_T/T})\cong \Gamma (\cC_T,\sL_T^{-5}\otimes w_{\cC_T/T}).$$
%Pulling back by $\eta: T\to \cpdp$, we obtain canonical isomorphism
%$$\eta^\ast \Ob_{{\cpdp}/\fdg}:=\eta^\ast H^1(\EE_{{\cpdp}/\fdg})\cong R^1\pi_{T\ast}(\sL_T^{-5}\otimes w_{\cC_T/T})\oplus  R^1\pi_{T\ast} \sL_T^{\oplus 5}.
%$$
For 
$$\mathring p\in H^1(\cC_T,\sP_{\cpgp})\and
\mathring {u}=(\mathring{u}_i)_{i=1}^5 \in H^1(\cC_T, \sL_{\cpgp}^{\oplus 5}),
$$
we define %denote $\delta=(p,\mathring{{u}_1}\mathring{{u}_2},\mathring{{u}_3},\mathring{{u}_4},\mathring{{u}_5})$ and 
\begin{eqnarray}\label{zeta}
\zeta(\mathring p,\mathring{u}):=5p\cdot\sum_{i=1}^5 {u}_i^4\cdot\mathring{{u}_i}+(\sum_{i=1}^5{u}_i^5)\cdot \mathring p,
%\and\zeta_2(\mathring p,\mathring{u}):=(\sum_{i=1}^5{\ti u}_i^5)\cdot \mathring p;
\end{eqnarray}
where $p$ and $u_i$ are the pull back of $\fp$ and $\fu_i$ to $\cC_T$, respectively.
The expression \eqref{zeta} is an element in $R^1\pi_{\cpgp\ast}(\omega_{\cC_{\cpgp}/\cpgp})\otimes_{\sO_{\cpgp}}\sO_T\cong \sO_T$.

One checks that this defines a homomorphism
$$\zeta:R^1\pi_{\cpdp \ast}\sL_{\cpdp}^{\oplus 5}\oplus R^1\pi_{\cpgp\ast} \sP_{\cpdp}\lra \sO_{\cpdp}.
$$

\begin{lemm}\label{same}
The two homomorphisms $\zeta$ and $\sigma_1$ coincide.
\end{lemm}

\begin{proof}
This follows from the explicit expression of $dh_1$ in affine coordinate generalizing the expression \eqref{dh}.
It is straightforward. 
\end{proof}

\begin{defi}
We define the degeneracy loci of $\sigma_1$ be 
$$D(\sigma_1)=\Bigl\{\xi\in\cpdp\,\big|\, \sigma_1|_\xi: \Ob_{{\cpgp}/\fdg}\otimes_{\sO_{\cpdp}}\kk(\xi)\lra \kk(\xi) \ \text{vanishes}\,\Bigr\}.
$$
\end{defi}

% Since the components $\phi_i$ of $\phi=(\phi_i)_{i=1}^5$ have no 
%common zeros, $p$ must be zero and also $\phi_1^5+\phi_2^5+\phi_3^5+\phi_4^5+\phi_5^5=0$, ie $\phi=[\phi_1,\phi_2,\phi_3,\phi_4,\phi_5]:C\to \Pf$ defines a map with image inside quintic threefold.
%\vsp
%
%
%Since ${\cpgp}\to\cpd$ is a cone-scheme, by taking its zero section we obtain an embedding $\cpd\to{\cpgp}$.

Following our convention, we dnote by $Q\sub\Pf$
the quintic threefold defined by $\sum x_i^5=0$. We
let $\bcM_g(Q,d)$ be the moduli of genus $g$ degree $d$ stable morphisms to $Q$.
Using $\bcM_g(Q,d)\sub\bcM_g(\Pf,d)$, we obtain embedding 
$$\bcM_g(Q,d)\sub \bcM_g(\Pf,d)\sub\cpgp,
$$
where the second inclusion is by assigning zero $p$-fields.

\begin{prop}\label{deg}
The degeneracy loci of $\sigma_1$ is $\bcM_g(Q,d)\sub \cpdp$; it is proper.
\end{prop}

\begin{proof}
Let $\xi=(C,L,\phi,p)\in \cpdp$, where $\phi=(\phi_i)_{i=1}^5\in H^0(C,L\upf)$.
The restriction of $\sigma_1=\zeta$ to $\xi$ takes the form
$\sigma_1|_{\xi}(\mathring p,\mathring \phi)=5p\sum \phi_i^4\mathring \phi_i+\sum \phi_i^5 \mathring p$.

Suppose $\sum \phi_i^5\ne 0$, then by Serre duality, we can find 
$\mathring \fp\in H^1(C, L^{-\otimes 5}\otimes \omega_C)$
so that $\mathring p\cdot \sum \phi_i^5\ne 0\in H^1(C,\omega_C)$. 
Letting $\mathring\phi_i=0$, we obtain $\sigma_1|_\xi\ne 0$.

Suppose $\sum\phi_i^5=0$ and $p\ne 0$. Then since $\phi_i$ have no common vanishing locus, 
for some $k$, $p\cdot \phi_k^4 \neq 0$. By Serre duality, we can find a $\mathring \phi_k$ so that
$p\cdot \phi_k^4\cdot\mathring \phi_k\ne 0\in H^1(C,\omega_C)$. By choosing other $\mathring\phi_i=0$, we obtain
the surjectivity of $\sigma_1|_\xi$.
This proves that the degeneracy loci (i.e. the non-surjective loci) of $\sigma_1$ is the collection of $(C,L,\phi,p)$
such that $\sum\phi_i^5=0$ and $p=0$. This set is $\bcM_g(Q,d)\sub\cpdp$.
\end{proof}

\subsection{The cosection factorizes}

Let $q:\cpdp\to \fdg$ be the tautological morphism. We form the distinguished triangle 
\beq\label{dt1}
q\sta \LL_{\fdg}\lra \LL_{\cpdp}\lra \LL_{\cpdp/\fdg}\mapright{\delta} q\sta \LL_{\fdg}[1].
\eeq
% gives a morphism 
% $$q\sta \TT_{\fdg}\to \TT_{\cpdp/\fdg}[1].$$
Composing $\phi_{\cpdp/\fdg}:\TT_{\cpdp/\fdg}\to \EE_{\cpdp/\fdg}$ with the dual of $\delta$ in the above
distinguished triangle, we obtain the morphism
$$\phi_{\cpdp/\fdg}\circ \delta\dual: q\sta \TT_{\fdg}\lra \TT_{\cpdp/\fdg}[1]\lra  \EE_{\cpdp/\fdg}[1].
$$ 
Denoting $\eta=H^0(\phi_{\cpdp/\fdg}\circ \delta\dual)$, we obtain the composite
\beq\label{tang}
\eta:  q\sta T_{\fdg}\lra
H^1(\TT_{\cpdp/\fdg})\lra H^1(\EE_{\cpdp/\fdg})=\Ob_{\cpdp/\fdg}.
\eeq
Following the construction in \cite[(4.3)]{KL}, the cokernel of (\ref{tang}) is the absolute obstruction sheaf of
$\cpdp$, which we denote by $\Ob_{\cpdp}$ . 

In this subsection, we show

\begin{prop}
\label{hahaha}
The cosection $\sigma_1:\Ob_{\cpdp/\fdg}\to \sO_\cpdp$ lifts to a  $\bar\sigma_1: \Ob_\cpdp\to \sO_\cpdp$.
\end{prop}

We continue to use the notation developed in the proof of Proposition \ref{obW}.

\begin{lemm}\label{cone}
The following composition is trivial:
$$
0=H^1(\sigma_1^\bullet\circ\phi_{\cpdp/\fdg}): H^1(\TT_{\cpdp/\fdg}) \mapright{}H^1(\EE_{\cpdp/\fdg})
\mapright{} 
R^1\pi_{\cpdp\ast}\omega_{\cC_{\cpdp}/\cpdp}.
$$
\end{lemm}

\begin{proof}
Using the universal curve $\pi_{\fdg}:\cC_{\fdg}\to\fdg$ of $\fdg$,
we introduce the direct image cone $\fC_\omega=C(\pi_{\ast}\omega_{\cC_{\fdg}/\fdg})$; we denote by
$\Vb(\omega_{\cC_{\fdg}/\fdg})$ the underlying bundle of $\omega_{\cC_{\fdg}/\fdg}$. 
Let $\cC_{\fC_\omega}=\cC_{\fdg}\times_{\fdg} \fC_\omega$ be the universal curve over $\fC_\omega$,
%and let $u_\cC:\cC_{\fC_\omega}\to \cC_{\fdg}$
 and $\pi_{\fC_\omega} :\cC_{\fC_\omega}\to \fC_\omega$ be the projection. 
%From the construction one has an canonical evaluation 
%Let $\be_\omega:\cC_{\fC_\omega}\to \Vb(\omega_{\cC_{\fdg}/\fdg})$ be the evaluation morphism.

%We also let $\overline{\fL}_\omega$ be the total space of the line bundle $u_\cC^\ast \omega_{/\fdg}$ then $\overline{\fL}_\omega=\Vb(\omega_{\cC_{\fdg}/\fdg})\times_{\cC_{\fdg}} \cC\om_g$. \\

Continue to denote by $(f_\cpdp, \cC_\cpdp, \sL_\cpdp)$ the universal family of $\cpdp$,
and using $\fu_i= f_\cpdp\sta x_i\in \Gamma(\cC_\cpdp,\sL_\cpdp)$ and
$%f_\cpdp=[\tilde u_1,\cdots,\tilde u_5]:\cC_\cpdp\to \Pf\and 
\fp \in \Gamma(\cC_\cpdp,\sP_\cpdp)$ 
the universal coordinate functions and $p$-field (cf. \eqref{pu}), we form 
$$\eps:=\fp\cdot (\fu_1^5+\ldots+ \fu_5^5)  \in\Gamma(\cC_\cpdp,\omega_{\cC_\cpdp/\cpdp}).
$$
It defines a morphism $\Phi_\eps :\cpdp\to \fC_\omega$
so that if we denote by $\ti \Phi_\eps: \cC_{\cpgp}\to \cC_{\fC_\omega}$ the tautological lift of $\Phi_\eps$ using
that both $\cC_\cpgp$ and $\cC_{\fC_\omega}$ are pull backs of $\cC_{\fdg}$, 
and denote by $\ee$ and $\ee'$ the evaluation morphisms as shown,  we
have a commutative diagram of morphisms of stacks over $\cC_{\fdg}$:
\beq\label{comm}
\begin{CD}
\cC_\cpdp@>{\ee}>> \Vb(\sL_{\fdg}\upf \oplus \sP_{\fdg})\\
@VV{\ti \Phi_\eps}V @VV{ h_1}V\\
\cC_{\fC_\omega}@>{\ee'}>>\Vb(\omega_{\cC_{\fdg}/\fdg}).
\end{CD}
\eeq
Here $h_1$ is defined in (\ref{co}).
This shows that the square below is commutative
\beq\label{comm2}
\begin{CD}
\pi_{\cpgp}\sta \TT_{{\cpdp}/{\fdg}} @= \TT_{\cC_{\cpdp}/\cC_{\fdg}}@>{}>> 
\ee^\ast\Omega\dual_{\Vb(\sL_{\fdg}\upf \oplus\sP_{\fdg} )/\cC_{\fdg}}\\
@VVV @VV{}V @VV{dh_1}V\\
\pi_{\cpgp}\sta \ti \Phi_\eps\sta  \TT_{{\fC_\omega}/{\fdg}}@= \ti \Phi_\eps\sta  \TT_{\cC_{\fC_\omega}/\cC_{\fdg}}@>{}>>
\ti \Phi_\eps\sta \ee^{\prime\ast}\Omega\dual_{\Vb(\omega_{\cC_{\fdg}}/\fdg)/\cC_{\fdg}}
\end{CD}
\eeq

Applying $R^1\pi_{\cpgp\ast}$ to the lower horizontal arrow we obtain the obstruction assignment homomorphism
\beq\label{3.11}
(0=)\ H^1(\Phi_\eps\sta\phi_{\fC_\omega/\fdg}): H^1(\Phi_\eps\sta  \TT_{\fC_\omega/{\fdg}})
\lra \Phi_\eps\sta R^1\pi_{\fC_\omega\ast} 
\omega_{\cC_{\fC_\omega}/\fC_\omega},
\eeq
which is trivial since $\fC_\omega$ is a vector bundle over $\fdg$ and $\cC_{\fC_\omega}\to \cC_{\fdg}$ is smooth. 

 % such that $\cC_\cpdp\cong\cpdp\times_{\fdg\om}\cC_{\fC_\omega}$. We denote the projection $\cC_\cpdp\to \cC_{\fC_\omega} $ by $v_\cC$.
 %a global evaluation map $\be_\cpdp:\cC_\cpdp\to \overline{\fL}_\omega$ that commutes with morphisms to $\cC_{\fC_\omega}$. 
Therefore, using the Cartesian squares
 \beq\label{huge}
 \begin{CD}
 %\cC_\cpdp@>{\be_\cpdp}>> \overline{\fL}_\omega@>{u_\fL}>>\Vb(\omega_{\cC_{\fdg}/\fdg})\\
 %@VV{||}V @VVV @VVV \\
 \cC_\cpdp@>{\ti \Phi_\eps}>> \cC_{\fC_\omega} \\ %@>{u_\cC}>> \cC_{\fdg}\\
 @VV{\pi_\cpdp}V @VV{\pi_{\fC_\omega}}V \\
 \cpdp@>{\Phi_\eps}>> \fC_\omega \\ %@>{u}>>\fdg,
 \end{CD}
 \eeq
and the commutativity of \eqref{comm2}, applying $R^1\pi_{\cpgp\ast}$, we see that
the composite
$$
H^1(\TT_{\cpdp/\fdg})\mapright{} R^1\pi_{\cpdp\ast}\ee^\ast\Omega\dual_{\Vb(\sL\upf_{\fdg}\oplus\sP_{\fdg} )/\cC_{\fdg}}
%\qquad\qquad
%$$
%$$\qquad\qquad\qquad\qquad\qquad\qquad\qquad\qquad\qquad
\lra R^1\pi_{\cpdp\ast}\ee^\ast  h_1^\ast \Omega\dual_{\Vb(\omega_{\cC_{\fdg}/\fdg})/\cC_{\fdg}}
$$
coincides with the composite
$$H^1(\TT_{\cpdp/\fdg})\lra H^1(\Phi_\eps\sta \TT_{\fC_\omega/\fdg})\mapright{0}
\Phi_\eps\sta R^1\pi_{\fC_\omega\ast}  \Omega\dual_{\Vb(\omega_{\cC_{\fdg}/\fdg})/\cC_{\fdg}}.
$$
Since the composite in the second line is trivial (cf. \eqref{3.11}), the composite in the first line is trivial.
Using 
$$\ee^\ast  h_1^\ast \Omega\dual_{\Vb(\omega_{\cC_{\fdg}/\fdg})/\cC_{\fdg}}
\cong \omega_{\cC_{\cpgp}/\cpgp},
$$
 this is exactly the vanishing desired by the Lemma, 
\end{proof}

\begin{proof}[Proof of Proposition \ref{hahaha}]
The composition of $\sigma$ with (\ref{tang}) is the $H^1$ of the composition 
$$\TT_{\fdg}[-1]\lra \TT_{\cpdp/\fdg}\mapright{\phi_{\cpdp/\fdg}}\EE_{\cpdp/\fdg}\mapright{\sigma_1^\bullet} R^\bullet\pi_{\cpdp\ast}\omega_{\cC_\cpdp/\cpdp},
$$
where the first arrow is the $\delta\dual$ in \eqref{dt1}. Lemma \ref{cone} implies the $H^1$ of the above composition is trivial.
\end{proof}

%\vsp
%
%A Corollary of the proof is the following obstruction reduction theorem. Let $\cU\sub \cpg$ be the open subset
%over which $\sigma_1: H^1(\EE_{\cpgp/\fdg})\to \sO_{\cpgp}$ is surjective. 
%We introduce the composite
%$$\sigma_1^\circ:\EE_{\cpgp/\fdg}|_\cU \lra  H^1(\EE_{\cpgp/\fdg}|_\cU)[1]\lra \sO_{\cU}[1].
%$$ 
%This is is a homomoprhism in the derived category since $\EE_{\cpgp/\fdg}$ concentrated at $[0,1]$.
%
%
%\begin{coro}
%\label{ob-red}\red 
%Let $\EE^{\mathrm{red}}_{\cU/\fdg}$ be the mapping cone $C(\sigma_1^\circ)$ of $\sigma_1^\circ$.
%Then the homomorphism $\phi_{\cpgp/\fdg}|_{\cU}$ lifts to 
%$\phi_{\cU/\fdg}^{\mathrm{red}}: \TT_{\cU/\fdg}\lra \EE^{\mathrm{red}}_{\cU/\fdg}$,
%and is a perfect relative obstruction theory. (cf. \cite[?]{KL}.)
%\end{coro}

Here we comment the background of this construction in Super-String Theories.
Let $K_\Pf$ be the total space of the canonical line bundle $\Pf$. The quintic polynomial $\sum x_i^5$
defines a regular map $\bw_\Pf\in \Gamma(\sO_{K_\Pf})$. Its critical locus is the quintic threefold $Q\sub \Pf$. In physics literature, the pair
$(K_\Pf, \bw_\Pf)$ is called a Landau-Ginzburg Model (non-linear).  In \cite{Sharpe},
Guffin and Sharpe constructed a path integral for genus zero A-twisted theory of the Landau Ginzburg space  
$(K_\Pf,\bw_\Pf)$ \cite{Sharpe}. In this paper, we have constructed a mathematical theory generalizing it to all genus. 
%The cosection comes from supersymmetry variation of $\bw_\Pf$:
%\beq\label{susy}
% \delta  \bw_\Pf=\delta [ \bp(\psi_1^5+\cdots+ \psi_5^5)]=5\bp \psi_i^4\delta \psi_i+\psi_i^5\delta \bp.
%\eeq
%where $[\psi_1,\cdots,\psi_5]$ gives a smooth map $\psi:C\to \Pf$ and $\bp$ is a smooth section
%of $\psi^\ast\sO(1)$ over $C$. The $\delta$ in (\ref{susy}) represents the supersymmetry operation.

 %There is an action of $\CC^\ast$ on the modui ${\cpgp}$. Let $\ep\in \CC^*$ then $\ep \cdot(C,L,\phi_i,p)=(C,L,\ep\phi_i,\ep^{-5}p)$. The action if fixed point free and the so the quotient is still a DM stack $\sS_h:=[{\cpgp}/\CC^\ast]$. \blue Here we use $\sS_h$ and call it Sharpe's stack. \black $\sS_h$ is isomorphic to the total space of the push down of the sheaf $\sO(-5)\otimes w_{\sC/M_g}$ from the universal family of $M_g(\Pf,d)$. Hence $\sS_h$ is fibered over $M_g(\Pf,d)$ with fiber equal to the vector space $H^0(C,L^{-\otimes 5}\otimes w_C)$ over each $\phi:C\to \Pf$. It is direct to check the cosection $\sigma$ is equivariant and hence descend to the cosection $\sigma':\Ob_{\sS_h/\fdg}\to \sO_{\sS_h}$.\\

\subsection{The virtual dimension}

We calculate the virtual dimension of $\cpgp$. %=\bcM_g(\Pf,d)^p$.
Let $\xi=(f,C, L, p)\in \cpgp$ be any closed point. The virtual dimension of $\cpgp/\fdg$ at $\xi$ is
$$\dim H^0(\EE_{\cpgp/\fdg}\otimes_{\sO_{\cpgp}}\kk(\xi))-\dim H^1(\EE_{\cpgp/\fdg}\otimes_{\sO_{\cpgp}}\kk(\xi)).
$$
By the expression of $\EE_{\cpgp/\fdg}$, the above term equals to
$$h^0(L\upf)+h^0(L^{-\otimes 5}\otimes \omega_C)-h^1(L\upf)-h^1(L^{-\otimes 5}\otimes \omega_C)=4-4g.
$$
Because 
$$\dim {\fdg} =\dim \fdg/\fM_g  + \dim \fM_g= (h^0(\sO_C)-1) +  3g-3=4g-4.
$$
The virtual dimension of $\cpgp$ at $\xi$ is zero.

\subsection{Localized virtual cycle}

We apply the theory developed in \cite{KL}. % (theorem $5.1$ in \cite{KL}) to get
%The homomorphism $\sigma_1: \Ob_{\cpgp}\to\sO_{\cpgp}$ induces a homomorphism
%$$\tilde\sigma_1: \EE_{\cpgp}\lra \red \sO_{\cpgp}[1].
%$$
We define a subcone-stack 
$$h^1/h^0(\EE_{\cpgp/\fdg})_{\sigma_1}\sub h^1/h^0(\EE_{\cpgp/\fdg})
$$
as follows. Let $\cU\sub \cpgp$ be the locus where $\sigma_1$ is surjective; we denote by
$$D(\sigma_1)=\cpgp-\cU
$$ 
its complement. Since $\sigma_1$ is surjective over $\cU$,
it induces a surjective bundle-homomorphism
\beq\label{sigu} \sigma_1|_\cU: h^1/h^0(\EE_{\cpgp/\fdg})\times_\cP\cU\lra \CC_\cU,
\eeq
where $\CC_\cU$ is the trivial line bundle on $\cU$.
We let $\ker(\sigma_1|_\cU)$ be the kernel bundle-stack of \eqref{sigu}; it is a codimension one
subbundle-stack of $h^1/h^0(\EE_{\cpgp/\fdg})\times_\cP\cU$.

We define
\beq\label{cone-stack}
h^1/h^0(\EE_{\cpgp/\fdg})_{\sigma_1}=\bl h^1/h^0(\EE_{\cpgp/\fdg})\times_{\cpgp} D(\sigma_1)\br \cup 
\ker(\sigma_1|_\cU).
\eeq
It is closed in $h^1/h^0(\EE_{\cpgp/\fdg})$. We endow it with the 
reduced structure. 
(We call \eqref{cone-stack} the kernel of $h^1/h^0(\EE_{\cpgp/\fdg})\to\CC_\cP$
induced by $\sigma_1$, where $\CC_\cP$ is the trivial line bundle on $\cP$.)

\begin{prop}The virtual normal cone cycle $[\bC_{\cpgp/\fdg}]\in Z\lsta h^1/h^0(\EE_{\cpgp/\fdg})$
lies inside $Z\lsta h^1/h^0(\EE_{\cpgp/\fdg})_{\sigma_1}$. 
\end{prop}

\begin{proof}
This is Proposition \cite[Thm 5.1]{KL}.
\end{proof}

In \cite{KL}, Kiem and the second named author constructed a localized Gysin map
%(Intersecting with the zero section of  $h^1/h^0(\EE_{\cpgp/\fdg})$):
$$0^!_{\sigma_1,\mathrm{loc}}: A\lsta h^1/h^0(\EE_{\cpgp/\fdg})_{\sigma_1}\lra A_{\ast-n} D(\sigma_1),
$$
where $-n$ is the rank of $\EE_{\cpgp/\fdg}$. 

%This loaclized Gysin map satisfies the usual property of the Gysin map by intersecting with
%the zero section of a cone-stack.  It is compatible with the Gysin map in that the following square
%is commutative:
%$$\begin{CD}
%A\lsta h^1/h^0(\EE_{\cpgp/\fdg})_{\sigma_1}@>>> A\lsta h^1/h^0(\EE_{\cpgp/\fdg})\\
%@VV{0^!_{\sigma_1,\mathrm{loc}}}V @VV{0^!}V\\
%A_{\ast-n} D(\sigma_1) @>>> A_{\ast-n} \cpgp
%\end{CD}
%$$
%By Proposition \ref{deg}, $D(\sigma_1)=\bcM_g(Q,d)$ is proper. 

\begin{defi-prop}\label{defi}
We define the localized virtual cycle of $(\cpgp,\sigma_1)$ be
$$[\cpgp]\virt_{\sigma_1}=[\bcM_g(\Pf,d)^p]\virt_{\sigma_1}:=0^!_{\sigma_1,\mathrm{loc}}([\bC_{\cpgp/\fdg}])\in A_0 
\bcM_g(Q,d).
%bcM_g(Q,d);
$$
We define the virtual enumeration $N_g(d)^p_\Pf:=\deg\, [\bcM_g(\Pf,d)^p]\virt_{\sigma_1}$.
 \end{defi-prop}

The number $N_g(d)^p_\Pf$ is the virtual counting of the Guffin-Sharpe-Witten Model 
$(\cpgp,\sigma_1)$. We call it the Gromov-Witten invariants of the 
moduli of stable morphisms to $\Pf$ with $p$-fields, or of the
Landau-Ginzburg space $(K_\Pf,\bw_\Pf)$.

\section{Degeneration of moduli of stable morphisms with $p$-fields}

In the second part, we will use degeneration to prove that $N_g(d)^p_\Pf$ coincides up to a sign with the
Gromov-Witten invariants $N_g(d)_Q$ of the quintic three-fold $Q$.

The degeneration we will use is to degenerate the moduli $\cP$ to the
moduli of stable morphisms to the normal bundle to $Q\sub \Pf$ coupled with $p$-field. 
After constructing a cosection of its obstruction sheaf,  the degeneration admits a 
localized virtual cycle that provides the proof of the
equivalence of two classes of invariants.

\subsection{The degeneration}

We let $V$ be the total space of the deformation of  $\Pf$ to the normal bundle of $Q\sub \Pf$; it
is the blowing up of $\Pf\times\Ao$ along $Q\times 0$, after taking out the proper transform of 
$\Pf\times 0$. Let 
\beq\label{q-proj}
q_{\Ao}: V\lra \Ao\and q_{\Pf}: V\lra\Pf
\eeq
be the two projections. Then the fiber of $V$ over $c\ne 0$ is
the $\Pf$, and the central fiber (over $0\in\Ao$) is the normal bundle $N$ to $Q\sub \Pf$.
We define the degree of a morphism $u: C\to V$ be
$\deg u=\deg(\rho\circ u)\sta\sO(1)$. 

We form the moduli of genus $g$ and degree $d$ stable morphisms 
$\barM_g(V,d)$. For the moment, we denote by 
$$(\ti f, \ti \pi): \ti \cC\lra V\times \bcM_g(V,d)
$$ 
the universal family of $\bcM_g(V,d)$.
Since $q_{\Ao}$ is proper away from the central fiber $N=V\times_{\Ao} 0$, and since $\Ao$ is affine,
the composite
$q_{\Ao}\circ \ti f: \ti \cC\to \Ao$ factors through a $\bcM_g(V,d)\to\Ao$. 
Its fiber over $c\ne 0\in\Ao$ are $\barM_g(\Pf,d)$; its central fiber is $\barM_g(\nq,d)$.
%Following our convention specified at the Introduction, we will abbreviate
%$$\cng=\bcM_g(N,d).
%$$

We now couple the stable morphisms with $p$-field. Let $\ti \sL=\ti f\sta\sO(1)$ and
$\ti \sP=\ti \sL^{-\otimes 5}\otimes \omega_{\ti \cC/\bcM_g(V,d)}$ be the tautological and auxiliary invertible sheaves. Like before, we define the moduli of stable morphisms coupled
with $p$-fields be
$$\cvgp\defeq \bcM_g(V, d)^p\defeq C(\ti \pi_{\ast}\ti \sP_{}),
$$
the direct image cone. It is over $\Ao$, and its fibers over $c\ne 0\in \Ao$ and $0\in\Ao$ are
$$\cvgp\times_{\Ao} c\cong \cpg,\quad \cvgp\times_{\Ao} 0:=\bcM_g(N,d)^p.
$$
Here $\bcM_g(N,d)^p$ is the moduli of stable morphisms to $N$ coupled with $p$-fields.

Following our convention, we denote by
\beq\label{fcv}(f_\cV,\pi_\cV): \cC_\cV\lra V\times\cV
\eeq
the universal map of $\cV$.

\subsection{The cone over $V$}
%\def\qq{\mathbb q}

%Change all $e_i$ to $\ti x_i$  all $b$ to $q$  all $b|_V$ to $q_\Pf$ \black
We construct the tautological cone $C(V)$ over $V$ that will be used to construct the evaluation
morphism $\ee_\fv$ of $\cC_{\cvgp}$. The evaluation map will be used to construct the
obstruction theory of $\cV$.

We let $B=\Vb(\sO(5))$ be the underlying line bundle of $\sO(5)$ over $\Pf$;  let  
$$ \bq_\Pf: B\times\Ao\to B\to \Pf\and \bq_{\Ao}:B\times \Ao\lra \Ao
$$
be the (composite of) projction(s). We let $t\in\Gamma(\sO_{\Ao})$ be the standard coordinate function of 
$\Ao$. We introduce tautological sections over $B\times\Ao$:
\beq\label{til-coor}
\ti x_i=\bq_\Pf\sta x_i\in \Gamma(\bq_{\Pf}^\ast \sO(1) ),\quad
\ti t=\bq_{\Ao}\sta t\in \Gamma(\sO_{B\times\Ao}),\and  \ti y\in \Gamma(\bq_{\Pf}^\ast \sO(5) ),
\eeq
where $\ti y$ is the section so that the morphism $B\times\Ao\to \Vb(\sO(5))$ induced by $\ti y$ is
the projection $B\times\Ao\to B=\Vb(\sO(5))$. (I.e. $\ti y$ is the pull back of the identity map $B\to B$.)

\begin{lemm}\label{tils}
We have a closed immersion
$$
V\cong (\ti s=0\}\sub B\times\Ao,\quad \ti s=\ti x_1^5+\ldots+\ti x_5^5-\ti t\cdot \ti y.
$$
\end{lemm}

\begin{proof}
%We let $V'=(\ti s=0)\sub B\times\Ao$. We show that $V'$ is canonically isomorphic to $V$. Indeed, 
We define 
$$\Phi: V-V\times_{\Ao}0\lra B\times\Ao
$$
via $\Phi\sta(\ti x_i)=q_\Pf\sta(x_i)$, $ \Phi\sta(\ti t)=q_{\Ao}\sta t$, and $\Phi\sta \ti y=t\upmo \cdot (x_1^5+\ldots+x_5^5)$,
where $q_\Pf: V\to \Pf$ is the projection, etc. (cf. \eqref{q-proj}). By definition, the image of $\Phi$ lies in
$\ti s=0$ and is an open immersion into $\ti s=0$. Using that $V$ is the deformation of
$\Pf$ to the normal cone of $Q\sub \Pf$, $\Phi$ extends to an isomorphism between $V$ and $\ti s=0$. 
This proves the Lemma.
\end{proof}

In the following, we will view $V\sub B\times \Ao$ using this isomorphism.
We next construct the cone $C(V)$ desired. We let $W_5=\CC\lAo$ (resp. $W_1=\CC\lAo^{\oplus 5}$)
be the trivial line bundle (resp. rank five trivial vector bundle) over $\Ao$.
We consider the rank six bundle
$$\pr_{\Ao}: W_1\times\lAo W_5 \mapright{}\Ao
$$
with the $\CC\sta$-action: $\CC\sta$ acts on the base $\Ao$ trivially and acts on fibers of $W_1$ (resp. $W_5$)
of weight one (resp. weight five). Namely, for $z\in W_1$ and $y\in W_5$,
$z^\sigma=\sigma z$ and $y^\sigma=\sigma^5y$.

We let $W_1\sta=W_1-0_{W_1}$, where $0_{W_1}$ is the zero section of $W_1$. 
We introduce 
$$C(V)=(\eps=0)\sub  W_1\sta\times_{\Ao} W_5, \quad \eps=z_1^5+\ldots+ z_5^5-t\cdot y.
$$
It is smooth and is $\CC\sta$-invariant.

We claim that $(W_1\sta\times\lAo W_5)/\CC\sta$ is isomorphic to $B=\Vb(\sO(5)$,
and under this isomorphism we have commuting (horizontal) quotient morphisms
\beq\label{quots}
\begin{CD}
W_1\sta\times\lAo W_5 @>{\Psi}>>B\\
@AA{\cup}A @AA{\cup}A\\
C(V) @>{/\CC\sta}>> V
\end{CD}
\eeq
Indeed, the top horizontal qoutient morphism follows from that of the weights of the $\CC\sta$-action on $W_1\times\lAo W_5$.
To see the full diagram, we construct explicitly the morphism $\Psi$ in \eqref{quots}.
%\beq\label{PHI}  \Psi: W_1\sta\times_{\Ao} W_5\lra B.
%\eeq
We let $U_i\sub \Pf$ be the open subset $x_i\ne 0$; we fix trivialization $\sO(5)|_{U_i}\cong \sO_{U_i}$
so that the transition function $\varphi_{ij}=x_i^5/x_j^5: \sO_{U_i}|_{U_i\cap U_j}\to \sO_{U_j}|_{U_i\cap U_j}$.
We define
$$\Psi_i: (W_1\sta-\{ z_i=0\})\times\lAo W_5\lra (B\times_\Pf U_i)\times \Ao
$$
via $(((z_i),y),t)\mapsto ((\frac{y}{z_i^5}, [z_1,\ldots,z_5]), t)$. This collection $\{\Psi_i\}$ form 
the morphism $\Psi$ in \eqref{quots}. 

By construction, $\Psi$ is $\CC\sta$-equivariant with $\CC\sta$ acting 
trivially on $B$, and factors to a $\CC\sta$-quotient morphism
$$\rho: C(V)\lra V.
$$

\vsp
For later purpose, we describe the tangent bundles $T_{C(V)/\Ao}$ and $T_{C(V)}$.
Using the defining equation of $C(V)$, they fit into the exact sequences
%$$0\lra r\sta\bq_\Pf\sta\sO(-5)\mapright{d\eps} \sO_{C(V)}^{\oplus 7}\lra \Omega_{C(V)}\dual\lra 0
%$$
%and
%$$
%0\lra r\sta\bq_\Pf\sta\sO(-5)\mapright{d\eps} \sO_{C(V)}^{\oplus 6}\lra \Omega_{C(V)/\Ao}\dual\lra 0, \quad (dt=0).
%$$
\beq\label{Omega}
0\lra T_{C(V)/\Ao}\lra  \sO_{C(V)}^{\oplus 5}\oplus \sO_{C(V)} \mapright{d'\eps} \sO_{C(V)}\mapright{} 0, \quad (d't=0),
\eeq
where $d'$ is the relative differential and $d'\eps|_{((z_i),y,t)}$ sends $((\mathring z_i), \mathring y)$ to $\sum 5z_i^4 \mathring z_i-t\mathring y$;
\beq\label{Omega-2}
0\lra T_{C(V)}\lra  \sO_{C(V)}^{\oplus 5}\oplus \sO_{C(V)} 
\oplus \sO_{C(V)} \mapright{d\eps} \sO_{C(V)}\mapright{} 0,
\eeq
where $d\eps|_{((z_i),y,t)}$ sends $((\mathring z_i), \mathring y,\mathring t)$ to 
$\sum 5z_i^4 \mathring z_i-t\mathring y-y\mathring t$. 

Together they fit into the exact sequence
\beq\label{OO}
0\lra T_{C(V)/\Ao}\lra T_{C(V)}\lra \sO_{C(V)}\lra 0.
\eeq

\subsection{The evaluation maps}
We now construct the evaluation morphism of $\cC_{\cvgp}$. 
Since $V$ is a family over $\Ao$, it is natural to construct the obstruction theory
of $\cV$ relative to $\fdg\times\Ao$.

To this purpose, we introduce $\tfdg=\fdg\times \Ao$, viewed as a stack over $\Ao$; denote by
$$\cC_\tfdg\defeq \cC_\fdg\times\Ao\lra {\fdg}\times\Ao=\tfdg\
$$ the universal 
curve, and denote by $\sL_\tfdg$ the pull back of $\sL_\fdg$ via $\cC_\tfdg\to\cC_\fdg$.

We form $\Vb(\sL_\tfdg\upf)\sta=\Vb(\sL_\tfdg\upf)-0_{\tfdg}$,
and consider the bundle over $\cC_{\fdg}$:
\beq\label{bundle}\Vb(\sL_\tfdg\upf)\sta\times_{\cC_\fdg} \Vb(\sL\of_\tfdg)\lra \cC_\tfdg.
\eeq
Note that for each $\xi\in \cC_\fdg$, the fibers of \eqref{bundle} over $\xi\times\Ao\sub \tfdg$ is
isomorphic to 
$$(L\upf -0)\times L\of \times \Ao\cong W_1\sta\times\lAo W_5,\quad L\defeq \sL_\fdg\upf\otimes_{\cC_\fdg}\kk(\xi),
$$
where the isomorphism is uniquely determined by an isomorphism $L\cong \CC$, and
two different isomorphisms are equivalent under a scaling of $(\CC^5-0)\times\CC$
by a $c\in \CC\sta$ with weights $(1,\ldots, 1,5)$ on the factors of $(\CC^5-0)\times\CC$.

We let $\CC\sta$ acts on the bundle \eqref{bundle} fiberwise with this weights.
We obtain the quotient $\Ao$-morphisms (the $\Ao$ is the base of $W\to \Ao$ and of $\tfdg=\fdg\times\Ao\to \Ao$)

$$%\beq\label{Psi}
\Vb(\sL_\tfdg\upf)\sta\times_{\cC_\tfdg} \Vb(\sL\of_\tfdg)\lra
\bl \Vb(\sL_\tfdg\upf)\sta\times_{\cC_\tfdg} \Vb(\sL\of_\tfdg)\br/\CC\sta\lra (W_1\sta\times\lAo W_5)/\CC\sta.
$$
We define
\beq\label{Zpri}
\cZ'= \bl \Vb(\sL_\tfdg\upf)\sta\times_{\cC_\tfdg} \Vb(\sL\of_\tfdg)\br\times_{(W_1\sta\times\lAo W_5)/\CC\sta} V;
\eeq
it is the preimage of $V\sub (W_1\sta\times\lAo W_5)/\CC\sta$ of the morphism above \eqref{Zpri}.
We define
\beq\label{ZZ}
\cZ=\cZ'\times_{\cC_{\tfdg}} \Vb(\sP_{\tfdg}).
\eeq

\vsp
We now construct the evaluation morphism
\beq\label{ev-V}
\ee_\fv: \cC_{\cvgp}\lra \cZ.
\eeq
We let $\sL_\cV=f_\cV\sta\sO(1)$, where $(f_\cV,\cC_\cV)$ is the universal family of $\cV$ (cf. \eqref{fcv}), let
$\sP_{\cvgp}=\sL_{\cvgp}^{-\otimes 5}\otimes\omega_{\cC_{\cvgp}/\cvgp}$ be the auxiliary invertible sheaf,
and let 
\beq\label{section}
\fp\in \Gamma(\cC_{\cvgp},\sP_{\cvgp}), \ 
\fu_i=f_{\cvgp}\sta \ti x_i\in \Gamma(\cC_{\cvgp}, \sL_{\cvgp})
\and \fy=f_\cvgp\sta \ti y \in\Gamma(\cC_{\cvgp},\sL_\cvgp\of)
\eeq
(cf. \eqref{til-coor}) be the universal $\fp$-field and the tautological coordinate functions.
Note that $(\cC_{\cvgp}, \sL_{\cvgp})$ induces an $\Ao$-morphism $\cvgp\to\tfdg$ so that $(\cC_\cV,\sL_\cV)$ is isomorphic
to the pull back of $(\cC_\tfdg,\sL_\tfdg)$.

Then the definition of $V\sub B\times\Ao$ implies that the sections in \eqref{section} satisfy
$$\fu_1^5+\fu_2^5+\fu_3^5+\fu_4^5 +\fu_5^5-t\cdot \fy=0,
$$
where $t$ is the coordinate function of $\Ao$ mentioned before.
Therefore the section $((\fu_i)_{i=1}^5,\fy, \fp)$ defines a section of 
$$\cZ\times_{\cC_\tfdg}\cC_{\cvgp}\lra \cC_{\cvgp}.
$$
This section induces a $\cC_{\cvgp}$-morphism % to the moduli of sections $\fS(\cZ)$:
$\cC_{\cvgp}\to \cZ\times_{\cC_\tfdg}\cC_\cvgp$. Composed with the projection $ \cZ\times_\tfdg\cvgp\to\cZ$, we obtain the evaluation morphism over $\cC_\tfdg$ in \eqref{ev-V}.

\subsection{The obstruction theory of $\cV/\tfdg$}

We will build the obstruction theories to carry out the degeneration for virtual cycles.
We first construct the relative obstruction theory of $\cV\to \tfdg$. The restriction of this obstruction theory
to fibers over $c\in\Ao$ will give the relative obstruction theories of $\cV_c=\cV\times\lAo c\to \fdg$.

We begin with a description of the tangent bundle $T_{\cZ'/\tfdg}$. Let $\varrho: \cZ'\to \tfdg$ be the
tautological projection.
Using the explicit description of $T_{C(V)/\Ao}$ given in \eqref{Omega}, 
and the construction of  $\cZ'$ in \eqref{Zpri},
we see that $\Omega\dual_{\cZ'/\tfdg}$ fits into the exact sequence
$$0\lra \Omega\dual_{\cZ'/\tfdg} \lra \varrho\sta \sL_{\tfdg}\upf\oplus \varrho\sta \sL_{\tfdg}\of\mapright{d\cE} \varrho\sta\sL_{\tfdg}\of\lra 0,
$$
where $d\cE$ restricted to $((z_i),y,t)\in \cZ'$ sends $((\mathring z_i), \mathring y)$ to $\sum 5z_i^4 \mathring z_i-t\mathring y$.
(cf. \eqref{Omega}.)
Using that $\sL_{\cvgp}=f_{\cvgp}\sta\sO(1)$, we obtain
\beq\label{HH0}
\ee_\fv\sta \Omega\dual_{\cZ'/\cC_\tfdg}\cong f_{\cvgp}\sta \sH\and
\ee_\fv\sta \Omega\dual_{\cZ/\cC_\tfdg}\cong f_{\cvgp}\sta \sH\oplus \sP_\cV,
\eeq
where $\sH$ on $B\times\Ao$ is defined by the exact sequence
$$0\lra \sH\lra q_{\Pf}\sta\sO(1)\upf\oplus q_\Pf\sta\sO(5)\mapright{d'\ti s} q_\Pf\sta\sO(5)\lra 0,
$$
where $d'\ti s$ is the differential of $\ti s$ in \eqref{tils}, after setting $d't=0$.
(Recall that $V\sub B\times \Ao$ by Lemma \ref{tils}.)

We have a similar description 
\beq\label{KK}
\ee_\fv\sta \Omega\dual_{\cZ'/\cC_\fdg}\cong f_{\cvgp}\sta \sK\and
\ee_\fv\sta \Omega\dual_{\cZ/\cC_\fdg}\cong f_{\cvgp}\sta \sK\oplus \sP_\cV,
\eeq
where $\sK$ is defined by the exact sequence
\beq\label{KK2}0\mapright{} \sK\mapright{i} q_{\Pf}\sta\sO(1)\upf\oplus q_\Pf\sta\sO(5)\oplus q_\Pf\sta\sO \mapright{d\ti s} q_\Pf\sta\sO(5)\lra 0,
\eeq
where $d'\ti s$ is the differential of $\ti s$ in Lemma \eqref{tils}.

\begin{prop}\label{ob-WV}
The pair $\cvgp\to\tfdg$ admits a perfect relative obstruction theory 
$$\phi_{\cvgp/\tfdg}: \TT_{\cvgp/\tfdg}\lra \EE_{\cvgp/\tfdg}:=
R^\ast\pi_{\cvgp\ast}(f_{\cvgp}^\ast \sH\oplus \sP_\cV). %\bl {f_{\cvgp}^\ast \sH}\oplus \sP_{\cvgp}\br.
$$
Its specialization at $c\ne 0\in \Ao$ (resp. $0\in \Ao$) give the perfect relative obstruction theory of
$\phi_{\cpgp/\fdg}$ (resp. $\phi_{\bcM_g(N,d)^p/\fdg}$).
\end{prop}

\begin{proof}
We fit $\ee_\fv:\cC_{\cvg}\to \cZ$ (cf. \eqref{ev-V}) into the commutative diagrams
\beq\label{first-V}
\begin{CD}
\cvg@<{\pi_\cvgp}<<\cC_{\cvgp} @>{\ee_\fv}>> \cZ \\
@VVV@VVV@VV{\pr}V\\
\tfdg@<{\pi_\tfdg}<<\cC_{\tfdg} @=\cC_{\tfdg},
\end{CD}
\eeq
where the left one is Cartesian. Using
\beq\label{bfV}
\pi_\cvg^\ast \TT_{\cvg/\tfdg}\cong \TT_{\cC_\cvg/\cC_{\tfdg}}\lra 
 \ee_{\cvg}^\ast \TT_{\cZ/\cC_{\tfdg}}= \ee_{\cvg}^\ast T_{\cZ/\cC_{\tfdg}}
 \eeq
and applying the projection formula, we obtain
\begin{equation}\label{def--V}
\phi_{\cvg/\tfdg}: 
\TT_{\cvg/\tfdg}\lra R^\bullet \pi_{\cvg\ast} \ee_{\cvg}^\ast T_{\cZ/\cC_{\tfdg}} .
\end{equation}

Let $\fS$ be the moduli of section of $\cZ\to \tfdg$ constructed in Subsection \ref{sec2.2}.
Because the evaluation morphism $\ee_\fv$ induces an open immersion $\cV\to \fS$,
using Proposition \ref{DEF} implies that $\phi_{\cvg/\tfdg}$ is a perfect relative obstruction theory.
%The Proposition follows from \eqref{HH0}. 
%$\ee_{\cvg}^\ast \TT_{\cZ/\cC_{\tfdg}} \cong f_{\cvgp}^\ast \sH \oplus \sP_{\cvgp}$.

Finally, the fiber product of every stack in (\ref{first-V}) with $c\ne 0\in \Ao$ gives the diagram used to construct $\phi_{\cpgp/\fdg}$.
Using $\iota_c: \cvgp\times_{\Ao}c \to \cvgp$, the functoriality of the construction ensures that
$\phi_{\cpgp/\fdg}$  is the composition of
$\TT_{\cpgp/\fdg}\to \iota_c\sta \TT_{\cvgp/\tfdg}$ with 
$$\iota_c\sta(\phi_{\cvgp/\tfdg}):\iota_c\sta \TT_{\cvgp/\tfdg}\lra \iota_c\sta \EE_{\cvgp/\tfdg}\cong\EE_{\cpgp/\fdg}.
$$ 
In case $c=0$, we define $\EE_{\bcM_g(N,d)^p/\fdg}\defeq \iota_0\sta \EE_{\cvgp/\tfdg}$.  This proves the Proposition.
\end{proof}

\subsection{The obstruction theory of $\cV/\fdg$}
\def\fr{\mathfrak r}
To compare the virtual cycle of $\cV_0$ with $\cV_{c\ne 0}$, we need the relative obstruction theory of $\cV\to\fdg$.

We using the $\phi_{\cV/\tfdg}$ just constructed.
We let 
\beq\label{B}
\sK\lra q_\Pf^\ast\sO_\Pf\cong \sO_{B\times \Ao}
\eeq
be the composition of $i$ in \eqref{KK} with the projection to the last factor.  
%The 
%above map gives 
%$$\sK|_V\lra \sO_{B\times\Ao}|_V=\sO_V.$$
%As like $\sH$ we will not distinguish $\sK$ from $\sK|_V$ by abuse of notation. 
%Since $R^\ast\pi_{\cvgp\ast}f_{\cvgp}^\ast{\sK}$ concentrates at $[0,1]$, we have the homomorphism to
We form
$$\mu: R^\ast\pi_{\cvgp\ast}f_{\cvgp}^\ast{\sK}\lra  R^\ast\pi_{\cvgp\ast}f_{\cvgp}^\ast\sO_{V}\lra
R^1\pi_{\cvgp\ast} \sO_{\cC_{\cvgp}}[-1],
$$ 
where the first arrow is $R^\bullet\pi_{\cV\ast}$ of \eqref{B}, and the second arrow if 
the tautological homomorphism from a two-term complex to its $H^1$.
 
We let $C(\mu\dual)$ be the mapping cone of $\mu\dual$, and let $C(\mu\dual)\dual$ be its dual.
It fits into the distinguished triangle 
\beq\label{dist0}
R^1\pi_{\cvgp\ast} \sO_{\cC_{\cvgp}}[-2]\lra C(\mu\dual)\dual\lra 
R^\ast\pi_{\cvgp\ast}f_{\cvgp}^\ast{\sK} \mapright{+1}  R^1\pi_{\cvgp\ast} \sO_{\cC_{\cvgp}}[-1].
\eeq
We define
$$\EE'_{\cvgp/\fdg}:= R^\ast\pi_{\cvgp\ast}(f_{\cvgp}^\ast{\sK}\oplus \sP_{\cvgp})
\and \EE_{\cvgp/\fdg}:=C(\mu\dual)\dual \oplus R^\ast\pi_{\cvgp\ast}\sP_{\cvgp}.
$$
Then one has
\beq\label{dist}
\begin{CD}
R^1\pi_{\cvgp\ast} \sO_{\cC_{\cvgp}}[-2]@>>>\EE_{\cvgp/\fdg}@>{\eta}>> 
\EE'_{\cvgp/\fdg} @>{\mu}>{+1}>  R^1\pi_{\cvgp\ast} \sO_{\cC_{\cvgp}}[-1].
\end{CD}
\eeq
By construction,
$\EE_{\cvgp/\fdg}$ is a derived object representable 
by a two-term complex of locally free sheaves; %its $H^0$ is identical to $H^0( R^\ast\pi_{\cvgp\ast}f_{\cvgp}^\ast{\sK})$;  
its $H^1$ is
$$ H^1(\EE_{\cvgp/\fdg})=\ker\{H^1(\mu): R^1\pi_{\cvgp\ast}(f_{\cvgp}^\ast{\sK} \oplus\sP_\cvgp )
\lra  R^1\pi_{\cvgp\ast} \sO_{\cC_{\cvgp}}\}.$$  Since (\ref{B}) is surjective, $H^1(\mu)$ is also surjective.
   
We now derive the perfect relative obstruction theory of $\cvgp\to\fdg$.
Substituting $\ti\fdg$ and $\ti\cC_{\fdg}$ in Proposition \ref{ob-WV} by $\fdg$ and $\cC_{\fdg}$ respectively,
and following the recipe in the proof of Proposition \ref{ob-WV}, we obtain a morphism
\beq\label{phi'}
 \phi'_{\cvgp/\fdg}:\TT_{\cvgp/\fdg}\lra R^\ast\pi_{\cvgp\ast}  \ee_\fv^\ast\TT_{\cZ/\cC_\fdg} % \oplus \sP_{\cvgp})
 \cong R^\ast\pi_{\cvgp\ast} (f_{\cvgp}^\ast{\sK}  \oplus \sP_{\cvgp}) 
 \defeq  \EE'_{\cvgp/\fdg}.
\eeq
%which satisfies the condition of perfect obstruction theory in \cite{BF} by applying proposition \ref{deformation} and 
Since moduli of sections of $\cZ\to\tfdg$
is isomorphic to the moduli of sections of $\cZ\to\fdg$, where $\cZ\to\fdg$ is via the composite $\cZ\to\tfdg\to\fdg$,
both are $\cvgp$, thus Proposition \ref{DEF} implies that
$\phi_{\cV/\fdg}'$ is a perfect relative obstruction theory.
 
According to Proposition \ref{ob-WV}, the obstruction sheaf of $\phi_{\cV/\fdg}'$ has an extra
factor $R^1\pi_{\cV\ast} \sO_{\cC_\cV}$ compared with that of $\phi_{\cP/\fdg}$ and of $\phi_{\bcM_g(N,d)^p/\fdg}$. 
Our solution is to lift it to a new obstruction theory (cf. \eqref{dist})
\beq\label{phi-new}
\phi_{\cvgp/\fdg}: \TT_{\cvgp/\fdg}\lra \EE_{\cvgp/\fdg}
\eeq
whose obstruction sheaf is parallel to that of $\phi_{\cP/\fdg}$ and of $\phi_{\bcM_g(N,d)^p/\fdg}$. 

We denote by $\ti\fr:\cC_\cvgp\to \cC_\tfdg$ the tautological morphism covering the tautological projection
$\fr$ shown in the Cartesian square
$$\begin{CD}
\cC_\cvgp@>{\ti\fr}>>\cC_{\tfdg}\\
@VV{\pi_\cvgp}V@VVV\\
\cV @>{\fr}>> \tfdg.
\end{CD}
$$
Applying $\TT_{\cdot/\cC_{\fdg}}$ to the evaluation $\cC_{\fdg}$-morphism $\ee_\fv: \cC_{\cvgp}\to \cZ$ (in  (\ref{first-V})),
the identity
$$\ti\fr=\pr\circ\ee_\fv: \cC_\cV\mapright{\ee_\fv} \cZ\mapright{\pr} \cC_{\tfdg}
$$
provides us a commutative square
$$
\begin{CD}
\ee_\fv^\ast \Omega\dual_{\cZ/\cC_{\fdg}}@>>> (\pr\circ \ee_\fv)^\ast \Omega\dual_{\cC_{\tfdg}/\cC_{\fdg}}\cong \sO_{\cC_{\cvgp}}\\
@AAA@AAA\\
 \TT_{\cC_{\cvgp}/\cC_{\fdg}}\cong \pi_\cvgp^\ast \TT_{\cvgp/\fdg}@>>> \ti\fr^\ast \Omega\dual_{\cC_\tfdg/\cC_\fdg}\cong \pi_\cvgp^\ast \fr^\ast \Omega\dual_{\tfdg/\fdg};
\end{CD}
$$
applying projection formula to both vertical arrows, we further obtain the commutative diagrams
\beq\label{HH}
\begin{CD}
\EE'_{\cvgp/\fdg}@>>> R^\ast\pi_{\cvgp\ast} \sO_{\cC_{\cvgp}}@>>>R^1\pi_{\cvgp\ast}\sO_{\cC_\cvgp}[-1]\\
@AA{\phi_{\cvgp/\fdg}'}A@AAA@AAA\\
\TT_{\cvgp/\fdg}@>>> \fr^\ast \Omega\dual_{\tfdg/\fdg}=\sO_{\cvgp}@>>>0.\\
\end{CD}
\eeq
This shows that $\mu\circ \phi_{\cV/\fdg}'=0$ (cf. $\mu$ is in \eqref{dist}).
Applying $\Hom(\TT_{\cvgp/\fdg},\cdot)$ to (\ref{dist}), we conclude that
the morphism $\phi_{\cvgp/\fdg}'$ (in \eqref{dist}) lifts (non-uniquely)
as stated in \eqref{phi-new}
such that
\beq\label{equation}
  \eta\circ \phi_{\cvgp/\fdg}=\phi'_{\cvgp/\fdg}.
\eeq

\begin{prop}\label{familyob}
The homomorphism $\phi_{\cvgp/\fdg}$ is a perfect relative obstruction theory of $\cvgp\to\fdg$.
%$$\phi_{\cvgp/\fdg}:\TT_{\cvgp/\fdg}\lra\ti\nu^\ast(R^\ast\pi_{\cvg\ast} f_{\cvg}^\ast {\sK})\oplus \ti\nu^\ast(R^\ast\pi_{\cvg\ast}\sP_{\ti Y}).$$
\end{prop}

\begin{proof}
We only need to check the criterion of perfect obstruction theory stated in the proof of Proposition \ref{deformation}.
Namely, we need to show that 
to any square zero extension $T\sub T'$ of affine schemes by $J$, and a commutative square
$$
\begin{CD}
T@>{\mm}>>\cvgp\\
@VVV @VVV\\
T'@>{\nn}>>\fdg,
\end{CD}
$$
the arrow $\phi_{\cV/\fdg}$ assigns an element $\varpi(\mm)\in H^1(T,\mm^\ast\EE_{\cvgp/\fdg}\otimes J)$
(cf. \eqref{obclass}) such that  there is a lifting $\mm':T'\to \cV$ of the square above
if and only if $\varpi(\mm)=0$.

Recall that $\phi_{\cV/\fdg}'$ is also a perfect relative obstruction theory. We let $\varpi(\mm)'\in 
H^1(T,\mm^\ast\EE'_{\cvgp/\fdg}\otimes J)$ be the associated obstruction class. Since $\phi_{\cV/\fdg}$ is a lift
of $\phi_{\cV/\fdg}'$, $\varpi(\mm)'$ is the image of $\varpi(\mm)$ under the homomorphism
$$H^1(\eta):
H^1(T,\mm^\ast\EE_{\cvgp/\fdg}\otimes J)\mapright{} H^1(T,\mm^\ast(R^\ast\pi_{\cvgp\ast}\EE'_{\cvgp/\fdg}\otimes J)
$$
induced by the $\eta$ in \eqref{dist}. Because of the distinguished triangle \eqref{dist},
$H^1(\eta)$ is injective. This proves that $\varpi(\mm)=0$ if and only if $\varpi(\mm)'=0$. Since the later is the obstruction
class, the former is too.

The other part of the criterion follows from the same reason. This proves the Proposition.
\end{proof}

\subsection{Comparison of obstruction theories}
\label{sec4.6}

Let $c\in \Ao$ be any closed point. We denote the restrictions to fibers over $c$ by 
$$\iota_c: \cV_c=\cV\times_{\Ao}c \mapright{\sub} \cV\and
\ee_{\fv_c}=\ee_\fv|_{\cC_{\cV_c}}: \cC_{\cV_c}=\cC_\cV\times_\cV\cV_c \lra\cZ_c=\cZ\times_{\Ao}c.
$$
Recall by Proposition \ref{ob-WV} that composing the tautological
$\TT_{\cV_c/\fdg}\to \iota_c\sta \TT_{\cvgp/\tfdg}$ with $\iota_c\sta\phi_{\cV/\tfdg}$ gives
the perfect relative obstruction theory 
$$\phi_{\cV_c/\fdg}: \TT_{\cV_c/\fdg}\lra \EE_{cV_c/\fdg}\defeq \iota_c\sta \EE_{\cV/\tfdg}.
$$
(Note that for $c\ne 0$, $\cV_c=\bcM_g(\Pf,d)^p$, and this obstruction theory coincide with the one constructed
in Subsection \ref{sec3.1}.)

We now compare the obstruction theory 
$\phi_{\cV/\fdg}$ with $\phi_{\cV_c/\fdg}$. %Let $\cZ_c=\cZ\times_{\Ao}c$; %denote by $\tau_c: {\cZ_c}\to \cZ$ the inclusion. 
Using the tautological exact sequence
%We have exact sequences of vector bundles over $\cZ_c$:
\beq\label{sequences}
0\lra T_{\cZ_c/\cC_\fdg}\lra  T_{\cZ/\cC_\fdg}|_{\cZ_c} \lra \sO_{\cZ_c}\lra  0,
%\begin{CD}
% 0@>>>T_{\cZ_c/\cC_\fdg}@>>> \Omega\dual_{\cZ/\cC_\fdg}|_{\cZ_c}@>>>\sO_{\cZ_c}@>>> 0\\
% @.@VV{||}V @VV{||}V @VV{||}V\\%  0@>>>T_{\cZ/\cC_\tfdg}|_{\cZ_c}@>>> \Omega\dual_{\cZ/\cC_\fdg}|_{\cZ_c}@>>>r_{\cZ}^\ast \Omega\dual_{\cC_\tfdg/\cC_\fdg}|_{\cZ_c}@>>> 0\\
% \end{CD}
\eeq
we obtain a morphism of distinguished triangles (the top line is an exact sequence of sheaves):
 \beq
 \begin{CD}
\ee_{\fv_c}^\ast \Omega\dual_{\cZ_c/\cC_\fdg}@>>>\ee_{\fv_c}^\ast \bl\Omega\dual_{\cZ/\cC_\fdg}|_{\cZ_c}\br@>>>
\ee_{\fv_c}^\ast \sO_{\cZ_c}\cong\sO_{\cC_{\cV_c}}@>{+1}>> 0\\
@AAA@AAA@AAA@.\\
 \TT_{\cC_{\cV_c}/\cC_\fdg}@>>>\TT_{\cC_\cvgp/\cC_\fdg}|_{\cC_{\cV_c}}@>>>\TT_{\cC_c/\cC_\cvgp}[1]@>{+1}>>
 \end{CD}
\eeq
%For simplicity we denote $M=\cvgp$ and $M_c=\cV_c$ so that $\pi_{M_c}:\cC_{\cV_c}\to\cV_c=M_c$ is the projection.

By projection formula, we have a morphism of distinguished triangles
\beq\label{KKP0}
\begin{CD} R^\ast\pi_{\cV_c\ast}\sO_{\cV_c}[-1]@>>>\EE_{\cV_c/\fdg}@>{\beta'}>>\EE'_{\cV/\fdg}|_{\cV_c}@>{+1}>>\\
@AAA@AA{\phi_{\cV_c/\fdg}}A@AA{\phi'_{\cV/\fdg}|_{\cV_c}}A@.\\
\TT_{\cV_c/\cV}@>>> \TT_{\cV_c/\fdg}@>{\gamma_0}>>\TT_{\cV/\fdg}|_{\cV_c}@>{+1}>>.
\end{CD}
\eeq 
Applying  the mapping  cone construction (\ref{dist}) to 
the top row of (\ref{KKP0}), and using  the octahedral axiom, we obtain a compatible diagram of mapping cones 
\beq\label{ver}
\begin{CD}
@.@.R^1\pi_{\cV_c\ast}\sO_{\cC_{\cV_c}}[-1]@>{=}>>R^1\pi_{\cV_c\ast}\sO_{\cC_{\cV_c}}[-1]\\
@.@.@AA{\mu|_{\cV_c}}A@AAA\\
R^\ast\pi_{\cV_c\ast}\sO_{\cC_{\cV_c}}[-1]@>>>\EE_{\cV_c/\fdg}@>{\beta'}>>\EE'_{\cV/\fdg}|_{\cV_c}@>>> R^\ast\pi_{\cV_c\ast}\sO_{\cV_c} \\
@AAA@AAA@AA{\eta|_{\cV_c}}A@AAA\\
\pi_{\cV_c\ast} \sO_{\cC_{\cV_c}}[-1]@>>>\EE_{\cV_c/\fdg}@>{\beta_0}>>\EE_{\cV/\fdg}|_{\cV_c}@>{\beta_1}>>\pi_{\cV_c\ast} \sO_{\cV_c}.
\end{CD}
\eeq
Restricting the perfect obstruction theory $\phi_{\cV/\fdg}$ (cf. Proposition \ref{familyob}) to $\cV_c$,
we obtain the following (not necessarily commuting) homomorphisms 
\beq\label{2nd}
\begin{CD}
\EE_{\cV_c/\fdg}@>{\beta_0}>>\EE_{\cV/\fdg}|_{\cV_c}\\
@AA{\phi_{\cV_c/\fdg}}A@AA{\phi_{\cV/\fdg}|_{\cV_c}}A\\
\TT_{\cV_c/\fdg}@>{\gamma_0}>>\TT_{\cV/\fdg}|_{\cV_c}.
\end{CD}
 \eeq

We consider 
$$\delta=\phi_{\cV/\fdg}|_{\cV_c}\circ  \gamma_0-\beta_0\circ\phi_{\cV_c/\fdg}:  
\TT_{\cV_c/\fdg}\lra\EE_{\cV/\fdg}|_{\cV_c}.
$$ 
Applying the commutative diagrams (\ref{equation}), (\ref{ver}) and (\ref{KKP0}), we conclude that
$$\eta|_{\cV_c}\circ\delta=\eta|_{\cV_c}\circ\phi_{\cV/\fdg}|_{\cV_c}\circ  \gamma_0 -\eta|_{\cV_c}\circ
 \beta_0\circ\phi_{\cV_c/\fdg}=\phi'_{\cV/\fdg}|_{\cV_c}\circ\gamma_0-\beta'\circ\phi_{\cV_c/\fdg}=0.
 $$  
Therefore, $\delta$ factors through $R^1\pi_{\cV_c\ast}\sO_{\cV_c}[-2]\to \EE_{\cV/\fdg}|_{\cV_c}$.

Because of this, after applying the truncation functor $\tau_{\le 1}$ to (\ref{2nd}),
we obtain a commutative square
\beq\label{1st}
\begin{CD}
 \EE_{\cV_c/\fdg}@>{\beta_0}>>\EE_{\cV/\fdg}|_{\cV_c}\\
 @AA{\phi^{\le 1}_{\cV_c/\fdg}}A@AA{\phi^{\le 1}_{\cV/\fdg}|_{\cV_c}}A\\
 \TT^{\leq 1}_{\cV_c/\fdg}@>{\gamma_0}>>\TT_{\cV/\fdg}|_{\cV_c}^{\leq 1}
\end{CD}
 \eeq 
On the other hand, applying the truncation functor $\tau_{\le 1}$ to the left square in (\ref{KKP0}),
we obtain another commutative square
\beq\label{KKP01}
 \begin{CD} \pi_{\cV_c\ast}\sO_{\cC_{\cV_c}}[-1]@>>>\EE_{\cV_c/\fdg}\\
  @AAA@AA{\phi^{\le 1}_{\cV_c/\fdg}}A\\
  \TT^{\leq 1}_{\cV_c/\cV}@>>> \TT^{\leq 1}_{\cV_c/\fdg}.
 \end{CD}
 \eeq  
Combined, we have a commutative diagrams
 \beq\label{KKPdiag}
 \begin{CD} 
 \sO_{\cV_c}[-1]\cong\pi_{\cV_c\ast}\sO_{\cC_{\cV_c}}[-1]@>>>\EE_{\cV_c/\fdg}@>{\beta_0}>>\EE_{\cV/\fdg}|_{\cV_c}@>{+1}>>\\
  @AAA@AA{\phi^{\le 1}_{\cV_c/\fdg}}A@AA{\phi^{\le 1}_{\cV/\fdg}|_{\cV_c}}A \\
    \TT^{\leq 1}_{\cV_c/\cV}@>>> \TT^{\leq 1}_{\cV_c/\fdg}@>{\gamma_0}>>\TT_{\cV/\fdg}|_{\cV_c}^{\leq 1}
 \end{CD}
\eeq  
% where the lower row is \blue NOT \black distinguisted triangle.
By (\ref{ver}) the top row is a distinguished triangle (but not the lower one).
%\footnote{\blue Applying  Proposition \ref{ob-WV} and Proposition \ref{familyob},
%we conclude that (\ref{KKPdiag}) is a truncated compatibility diagram \eqref{KKPdiag}
%that satisfies the conditions Lemma \ref{weakKKP}  in Appendix.
%
%If a cosection
% factorizes $H^1$ of the upper row of (\ref{KKPdiag})
% \beq\label{finally}
% \sO_{\cV_c}\lra\Ob_{\cV_c/\fdg}\lra \Ob_{\cV/\fdg}|_{\cV_c}\lra 0
% \eeq
% then we can conclude the family cosection localized virtual cycle of $\cV$ specialized to the cosection localized virtual
% cycle of each  $\cV_c$ for any $c$. Note here by the second row of (\ref{sequences}) we can interprete the first arrow of (\ref{finally}) $\sO_{\cV_c}\lra\Ob_{\cV_c/\fdg}$ is 
% 
%$$\sO_{\cV_c}\cong (\kappa^\ast \Omega\dual_{\tfdg/\fdg})|_{\cV_c}\lra \Ob_{\cV/\tfdg}|_{\cV_c}\cong\Ob_{\cV_c/\fdg}$$
% where $\kappa:\cvgp\to\tfdg$ and $\kappa|_c:\cV_c\to \tfdg$ are the canonical projection
% and the map $$\kappa^\ast \Omega\dual_{\tfdg/\fdg}\lra \Ob_{\cV/\tfdg}$$
% is the morphism associated from the composition of the inclusion $T_{\tfdg/\fdg}\sub \Omega\dual_{\tfdg}$  with the canonical $$\kappa^\ast \Omega\dual_{\tfdg}[-1]\to \TT_{\cV/\tfdg}\mapright{\phi_{\cV/\tfdg}}\EE_{\cV/\tfdg}.$$
% 
%}
 
We comment that applying results in \cite{Kresch2}, this diagram implies that the virtual cycles of $\cV_c$ is the pull back
via $\iota_c: \cV_c\to\cV$ of the virtual cycle of $\cV$. In our case, we are using localized virtual cycles
via cosections of the obstruction sheaves, thus we need to construct a cosection of the obstruction sheaf
$$\Ob_{\cV}=\coker\{ T_{\fdg}\otimes_{\sO_\fdg}\sO_\cV \lra H^1(\EE_{\cV/\fdg})\}.
$$

\subsection{Family cosection}
 
We first construct a cosection of the obstruction sheaf $\Ob_{\cvgp/\tfdg}$. The construction is parallel to the
case $\cpgp=\bcM_g(\Pf,d)^p$.  

First, we define a bi-linear morphism of bundles
$$h: \Vb(\sL_{\tfdg}\upf\oplus \sL_{\tfdg}\of\oplus \sP_{\tfdg})\lra 
\Vb(\sL_{\tfdg}\of)\times_{\cC_{\tfdg}}\Vb(\sP_{\tfdg})\lra \Vb(\omega_{\cC_{\tfdg}/\tfdg}).
$$
Here the first arrow is $(\pr_2,\pr_3)$, where $\pr_i$ is the $i$-th projection; the second arrow is induced by 
tensoring of sheaves of $\sO_{\cC_{\tfdg}}$-modules 
$\sL_{\tfdg}\of\otimes\sP_{\tfdg}\to\omega_{\cC_{\tfdg}/\tfdg}$. Using that the family $\cZ\to \cC_{\tfdg}$ in \eqref{ZZ} 
is a subfamily
$$\cZ\sub   \Vb(\sL_{\tfdg}\upf)\sta \times_{\cC_{\tfdg}}
 \Vb(\sL_{\tfdg}\of)\times_{\cC_{\tfdg}}
  \Vb( \sP_{\tfdg})
  \sub \Vb(\sL_{\tfdg}\upf\oplus \sL_{\tfdg}\of\oplus \sP_{\tfdg}),
$$
composing with $h$, we obtain a $\cC_{\tfdg}$-morphism
\beq\label{Z-eta}
\cZ\lra \Vb(\omega_{\cC_{\tfdg}/\tfdg}).
\eeq

\begin{lemm}
The homomorphism \eqref{Z-eta} induces a homomorphism
$$\sigma^\bullet: \EE_{\cvgp/\tfdg}\lra  R^1\pi_{\cvgp\ast} \sO_{\cC_{\cvgp}}[-1]
$$
whose restriction to $\cvgp\times_{\Ao} c\cong \cpgp$, $c\ne 0$, is  proportional (by an element in $\CC\sta$) 
to $\sigma_1^\bullet$ in 
\eqref{si-cosection}.
\end{lemm}

\begin{proof}
The proof is exactly as in Section 3.2. We will omit it here.
\end{proof}

We denote
\beq\label{si-V}
\sigma=H^1(\sigma^\bullet): \Ob_{\cvgp/\tfdg}\defeq H^1(\EE_{\cvgp/\tfdg})\lra R^1\pi_{\cvgp\ast}\omega_{\cC_{\cvgp}/\cvgp}
\cong \sO_{\cvgp}.
\eeq

\vsp

Let $\ti \fq:\cvgp\to \tfdg$ be the projection.  The distinguished triangle
$\ti\fq^\ast\LL_{\tfdg}\to\LL_{\cvgp}\to\LL_{\cvgp/\tfdg}\to\ti\fq^\ast\TT_{\tfdg}[1]$
gives a morphism 
$\ti\fq^\ast\TT_{\tfdg}\to \TT_{\cvgp/\tfdg}[1]$, which 
composed with $\phi_{\cvgp/\tfdg}:\TT_{\cvgp/\tfdg}\to \EE_{\cvgp/\tfdg}$ gives
$$\eta: \ti\fq^\ast\TT_{\tfdg}\lra  \EE_{\cvgp/\tfdg}[1].
$$ 
Taking the cokernel of  the $H^0$ of this arrow, we obtain the absolute obstruction sheaf
\beq\label{tangf}
\Ob_{\cvgp}\defeq \coker\{ H^0(\eta): \ti\fq^\ast\Omega_{\tfdg}\dual\lra  H^1(\EE_{\cvgp/\tfdg})\}.
\eeq

\begin{lemm}\label{family-cone}
Th following composite vanishes
\beq\label{familyseq}
\ti\fq\sta\Omega_{\tfdg}\dual \mapright{H^0(\eta)} H^1(\EE_{\cvgp/\tfdg})\mapright{\sigma} 
R^1\pi_{\cvgp\ast}\omega_{\cC_{\cvgp}/\cvgp}.
\eeq
\end{lemm}

\begin{proof}The proof is exactly the same as the that of Proposition \ref{hahaha},
and will be omitted.
\end{proof} 

This immediately gives

\begin{coro}\label{hahahaha}
The cosection $\sigma:\Ob_{\cvgp/\tfdg}\to \sO_{\cvgp}$ lifts to a cosection $\bar\sigma: \Ob_{\cvgp}\to \sO_{\cvgp}$.
\end{coro}

%\footnote{\yellow
%This is because, due to the change of obstruction theory.[fill in the details ...]
%We have a diagram of exact sequences (each row and column) by proposition \ref{ob-WV}
% 
% 
%\beq\label{big}
%\begin{CD}
%@. \ti\fq^\ast\Omega_{\tfdg/\Ao}\dual@>{=}>> (w'_V)^\ast\Omega_{\fdg}\\
%@.@VVV@VVV\\
%\sO_{\cvgp}@>{\delta}>>\Ob_{\cvgp/\tfdg}@>>>\Ob_{\cvgp/\fdg}@>>>0\\
%@VV{||}V@VVV@VVV\\
%\sO_{\cvgp}@>>>\Ob_{\cvgp/\Ao}@>>>\Ob_{\cvgp}@>>>0\\ 
%@.@VVV@VVV\\
%@.0@. 0.
% \end{CD}
%\eeq
%}

Lastly, we describe the degeneracy (non-surjective) loci of $\sigma$.
As before, we say $\sigma$ is degenerate at $\xi\in\cvgp$ if $\sigma|_\xi$ is not
surjective (i.e. is trivial).
Let $\xi\in\cvgp$ be any closed point; $\xi$ is represented by
$$((\phi_i),b,p)\in  H^0(L\upf)\times H^0(L\of)\times H^0(L^{-\otimes 5}\otimes \omega_C)
$$ 
for $(C,L)\in \fdg$ the point under $\xi$.
Then $\sigma|_\xi: \Ob_{\cvgp/\tfdg}|_\xi %\otimes_{\sO_{\cvgp}}\kk(\xi)
\to \CC$
is identical to the composite of the inclusion
$$\Ob_{\cvgp/\tfdg}|_\xi\sub H^1(L\upf)\oplus H^1(L\of)\oplus H^1(L^{-\otimes 5}\otimes\omega_C)
$$
with the pairing
$$H^1(L\upf)\oplus H^1(L\of)\oplus H^1(L^{-\otimes 5}\otimes\omega_C)\lra H^1(\omega_C)
$$
defined via $((\mathring \phi_i), \mathring b, \mathring p)\mapsto  \mathring b\cdot p+b\cdot \mathring p$.
Like the proof of Proposition \ref{deg}, this description shows that the degeneracy loci of $\sigma$
is $\bcM_g(Q,d)\times\Ao\sub \cvgp$, where the inclusion is via vanishing $p$-fields and the inclusion
$\bcM_g(Q,d)\times\Ao\sub \bcM_g(V,d)$ induced by the tautological inclusion $Q\times \Ao\sub V$.

\begin{lemm}\label{degWV}
The degeneracy loci of the cosection $\bar\sigma$ is $\bcM_g(Q,d)\times \Ao \sub\cvgp$; it is proper over $\Ao$.
\end{lemm}

\begin{proof}
We first need to verify that $\sigma$ is as given. The proof of this is exactly the same as that of 
Lemma \ref{same}. Using this
description, we argue that the degeneracy loci of the cosection $\sigma:\Ob_{\cvgp/\fdg}\to\sO_{\cvgp}$ is 
$\bcM_g(Q,d)\times \Ao \sub\cvgp$; thus is proper over $\Ao$. Since $\bar\si$ is a lift of $\sigma$, the degeneracy
loci of $\bar\si$ coincides with that of $\sigma$. This proves the Lemma.
\end{proof}

\subsection{The constancy of the invariants}

By direct verification, the virtual dimension of $\cvgp$ is one.
Using Lemma \ref{degWV} and Corollary \ref{hahahaha}, following the convention introduced
in Subsection 3.6, we denote by 
$$h^1/h^0(\EE_{\cV/\fdg})_{\bar\sigma}\sub h^1/h^0(\EE_{\cV/\fdg})
$$
the kernel of a cone-stack morphism $h^1/h^0(\EE_{\cV/\fdg})\to\CC_\cV$
induced by $\bar\sigma$ defined as in 
\eqref{cone-stack}.\footnote{ It is $h^1/h^0(\EE_{\cV/\fdg})$ along the degeneracy loci and is the kernel of
$h^1/h^0(\EE_{\cV/\fdg})\to\sO_{\cvgp}$ induced by $\sigma$ away from the degeneracy loci.}

Let 
$$[\bC_{\cpgp/\fdg}]\in Z\lsta h^1/h^0(\EE_{\cV/\fdg})
$$
be the intrinsic normal cone embedded using the obstruction theory $\phi_{\cvgp/\fdg}$.
Because of Lemma \ref{degWV} and Corollary \ref{hahahaha}, applying \cite[Thm 5.1]{KL} 
we conclude that
$$[\bC_{\cpgp/\fdg}]\in Z\lsta h^1/h^0(\EE_{\cV/\fdg})_{\bar\sigma}.
$$
We then applying the localized Gysin map \cite{KL}
$$0^!_{\bar\sigma,\mathrm{loc}}: A\lsta h^1/h^0(\EE_{\cV/\fdg})_{\bar\sigma}\lra A\lsta (\bcM_g(Q,d)\times\Ao).
$$

\begin{defi}\label{family-cycle} We define the localized virtual cycle of $(\cvgp,\bar\sigma)$ be
$$[\cvgp]\virt_{\bar\sigma}:=0^!_{\bar\sigma,\mathrm{loc}}([\bC_{\cvgp/\fdg}])\in A_1 (\bcM_g(Q,d)\times\Ao).
$$
 \end{defi}

Now let $c\in \Ao$ be any closed point and let $\jmath_c: c\to \Ao$ be the closed inclusion. We denote $\cN:=\cV\times_{\Ao} 0$.
By the compatibility stated in diagram (\ref{KKPdiag}) and Corollary (\ref{hahahaha}), we apply \cite[Thm 5.3]{KL} to obtain

\begin{prop}\label{shriek}
Under the shriek operation of cycles ($c\ne 0$),
$$\jmath_c\sta ([\cvgp]\virt_{\bar\sigma})=[\cpgp]\virt_{\bar\sigma_1}\in  A_0 \bcM_g(Q,d),
\quad 
\jmath_0^\ast([\cvgp]\virt_{\bar\sigma})=[\cN]\virt_{\bar\sigma_0}\in  A_0 \bcM_g(Q,d).
$$
\end{prop}
 
Here $[\cN]\virt_{\bar\sigma_0}$ is the localized virtual cycle using the obstruction theory of
$\cN$ induced by the restricition of $\phi_{\cvgp/\tfdg}$(Prop \ref{ob-WV}) and the restriction of the cosection 
$\bar\sigma_0=\bar\sigma|_{\cN}$. % to be clarified momentarily. \black
%\Ob_{\bcM_g(N,d)^p}\to\sO_{\cngp}.
%$$

\section{Gromov Witten invariant of $(K_{N},\bw_{\nq})$}
\def\lcng{_{\cng}}

We continue to denote by
$$r: \nq\lra Q
$$ 
the normal bundle to $Q$ in $\Pf$. 
Let $K_{\nq}$ be the total space of the canonical line bundle of $\nq$, which is isomorphic
to the underlying line bundle of the pull back $r\sta \sO(-5)$.
The duality paring $\sO_Q(5)\otimes_{\sO_Q}\sO_Q(-5)\to\sO_Q$ defines a
regular function $\bw_{\nq}\in \Gamma(\sO_{K_{\nq}})$. The degree $\deg [\cN]_{\bar\sigma_0}\virt$ are the Gromov-Witten
invariants
of the Landau-Ginzburg space $(K_{N},w_{\nq})$. % we are focusing in this section.

We denote by $\bcM_g(\nq,d)$ the moduli space of genus $g$ degree $d$ stable morphisms to $\nq$,
where the degree is measured by their images in $\Pf$ via $\nq\to Q\sub \Pf$. Because 
$N=V\times\lAo 0$, canonically $\bcM_g(N,d)=\cV\times\lAo 0$.
The moduli of stable maps coupled with $p$-fields 
is identical to $\cN$
$$\cN:=\cV \times_{\Ao} 0=\bcM_g(V,d)^p\times\lAo 0\cong \bcM_g(N,d)^p.$$
We let 
$$(f_{\cN},\pi_{\cN}): \cC_{\cN}\lra N\times \cN
$$ 
be the universal map of $\cN$. By definition, it is the restriction of $(f_\cV, \pi_\cV, \cC_\cV)$ to 
the fiber over $0\in \Ao$.
%let $\sL_{\cng}=f\lcng\sta\sO(1)$ and $\sP\lcng=\sL\lcng^{-\otimes 5}\otimes\omega_{\cC\lcng/\cng}$
%be its tautological and auxiliary invertible sheaves.
%We define the moduli with $p$-fields $\bcM_g(N,d)^p$ be the direct image cone
%$$\cngp\defeq \bcM_g(N,d)^p\defeq C(\sP\lcng).
%$$
%We will use cosection localization to contruct its virtual fundamental class.

\subsection{The invariants and the equivalence}
\def\cqg{\cQ}

As indicated in the beginning of Subsection \ref{sec4.6}, 
we have evaluation morphism 
$$\ee_{\cN}: \cC_\cN\lra \cZ_0=\cZ\times\lAo 0.
$$
 By construction in Proposition \ref{ob-WV},
\beq\label{ob-cngp}
\phi_{\cngp/\fdg}: \TT_{\cngp/\fdg}\lra R^\ast\pi_{\cngp\ast} \ee_\cN^\ast \TT_{\cZ_0/\cC_{\fdg}}
%\boxplus e_{\cngp}^\ast \TT^\bullet_{\bL_P/\cC_{\fdg}})\cong
:=\EE_{\cngp/\fdg}
\eeq
is a perfect relative obstruction theory of $\cngp/\fdg$, which is identical to
the restriction of $\phi_{\cV/\tfdg}$ to the fiber over $0\in \Ao$.
%gives the deformation theory of $\cngp/\fdg$:
%$\beta^\ast \pi_{\cqg\ast} \sP_\cqg$
% over $\cng$. Namely we form
% $$\cngp=C(\pi_{\cqg\ast}\sP_\cqg)\times_\cqg \cng.$$
% We denote $\upsilon:\cngp\to \cqg$ and $\nu_Q:\cngp\to \cng$ to be the tautological projections.
%     
%
% We consider the deformation theory of $ \cngp/\fdg$. The argument analogous to lemma \ref{YQfib} shows
%

We let 
$\sigma_0$ be the restriction of $\sigma$ to $\cN$:
\beq\label{sig-o}\sigma_0=\sigma|_{\cN}: \Ob_{\cN/\fdg}\lra \sO_\cN.
\eeq

\begin{prop}\label{degQ}
The cosection $\sigma_0$ lifts to a cosection $\bar\sigma_0:\Ob_\cN\to \sO_\cN$.
The degeneracy (non-surjective) loci of the cosection $\bar\sigma_0$ is $\barM_g(Q,d)\sub \barM_g(\nq,d)^P$; thus is proper.
\end{prop}

\begin{proof}
This follows directly from Lemma \ref{degWV}.
\end{proof}
 %There is an action of $\CC^\ast$ on the modui $\cngp$. Let $\ep\in \CC^*$ then $\ep \cdot(C,L,\psi_i,p)=(C,L,\ep\psi_i,\ep^{-5}p)$. The action if fixed point free and the so the quotient is still a DM stack $\sS_h:=[\cngp/\CC^\ast]$. \blue Here we use $\sS_h$ and call it Sharpe's stack. \black $\sS_h$ is isomorphic to the total space of the push down of the sheaf $\sO(-5)\otimes w_{\sC/M_g}$ from the universal family of $M_g(\Pf,d)$. Hence $\sS_h$ is fibered over $M_g(\Pf,d)$ with fiber equal to the vector space $H^0(C,L^{-\otimes 5}\otimes w_C)$ over each $\psi:C\to \Pf$. It is direct to check the cosection $\sigma$ is equivariant and hence descend to the cosection $\sigma':\Ob_{\sS_h/\fdg}\to \sO_{\sS_h}$.\\

Because the virtual dimension of $\cV$ is one, the virtual dimension of $\cN$ is $0$. 
By Proposition \ref{deg}, applying cosection localization Gysin map \cite[Theorem 5.1]{KL}, we obtain 

 \begin{defi-prop} We define the localized virtual cycle of $\barM_g(\nq,d)^P$ be
$$[\barM_g(\nq,d)^P]\virt_\sigma:=0^!_{\sigma,\mathrm{loc}}([\bC_{\barM_g(\nq,d)^P/\fdg}])\in A_0 \barM_g(Q,d);
$$
we denote $N_g(d)_{K_{N_Q}}=\deg \, [\barM_g(\nq,d)^P]\virt_\sigma$.
 \end{defi-prop}

We call $N_g(d)_{K_{N_Q}}$ the formal Landau-Ginzburg Model.

\begin{theo}\label{thm5.3}
For any positive $d$, the invariants coincide: $N_g(d)^p_\Pf=N_g(d)_{K_{N_Q}}$.
\end{theo}

\begin{proof}
It follows directly from Proposition \ref{shriek}.
\end{proof}

\subsection{Comparing with the GW invariant of Quintics}
We now show that the formal Landau-Ginzburg model gives the same invariants as the Gromov-Witten invariants of $Q$ 
up to signs. 

We first construct a perfect relative obstruction theory of $\cngp\to \cqg$. In the fiber product over $\fdg$
\beq\label{v-v}
\begin{CD}
\cngp@>{\gamma}>> \fC\defeq C(\pi_{\fdg\ast}(\sL_\fdg\of\oplus\sP_\fdg))\\
@VV{v}V@VVV\\
\cqg\defeq \bcM_g(Q,d) @>>>\fdg,
\end{CD}
\eeq
where $C(\pi_{\fdg\ast}(\sL_\fdg\of\oplus\sP_\fdg))$ is the direct image cone constructed in Subsection 2.1,
the morphism $\gamma$ pulls back the relative perfect obstruction theory 
\begin{equation}\label{relll}
 \TT_{{\fC}/\fdg}\lra  \EE_{{\fC}/\fdg} 
\end{equation}
to the morphism
$$\phi_{\cngp/\cqg}: \TT_{\cngp/\cqg}\lra\EE_{\cngp/\cqg}:=\gamma^\ast \EE_{{\fC}/\fdg}.
$$
 
By Proposition \ref{deformation}, $\phi_{\cngp/\cqg}$ is the perfect relative obstruction theory 
associated with the direct image cone stack $\cngp\cong C(\pi_{\cqg\ast}(\sL_\cqg\of\oplus\sP_\cqg))$ relative to $\cqg=\bcM_g(Q,d)$. 
 
We define
\beq\label{fQ}
\fQ=  \Vb(\sL_\fdg\upf)\sta\times_\Pf Q.
\eeq 
The evaluation maps of $\cngp$ and $\cqg$ fit into the diagram
$$
\begin{CD}
\cC_\cngp@>{\be_\cngp}>>\cZ_0@>>>\Vb(\sL_{\fdg}\of)\times_{\fdg}\Vb(\sP_{\fdg})\\
@VV{\upsilon_\cC}V @VVV@VVV\\
\cC_\cqg@>{\be_\cqg}>> \fQ@>>>\cC_\fdg
\end{CD}
$$
where the right square is a fiber product of smooth morphisms, and $\upsilon_\cC$ is induced by 
the vertical arrow $v$ in diagram \eqref{v-v}. 

The diagram associates a morphism between distinguished triangles
$$
\begin{CD}
\be_\cngp^\ast T_{\cZ_0/\fQ}@>>>\be_\cngp^\ast T_{\cZ_0/\cC_\fdg}@>>>\upsilon_\cC^\ast \be_\cqg^\ast T_{\fQ/\cC_\fdg}@>{+1}>>\\
@AAA@AAA@AAA\\
\TT_{\cC_\cngp/\cC_\cqg}@>>>\TT_{\cC_\cngp/\cC_\fdg}@>>> \upsilon_\cC^\ast \TT_{\cC_\cqg/\cC_\fdg}@>{+1}>>\\
\end{CD}
$$  
\black
 
Denoting $\EE_{\cngp/\cqg}:=\pi_{\cngp\ast}\be_\cngp^\ast T_{\cZ_0/\fQ}$, then by the projection formula we have
\beq\label{triang}
\begin{CD}
\EE_{\cngp/\cqg}@>>>\EE_{\cngp/\fdg}@>{h}>>\upsilon^\ast \EE_{\cqg/\fdg}@>{+1}>>.\\
@AA{\phi_{\cngp/\cqg}}A@AA{\phi_{\cngp/\fdg}}A@AA{\upsilon^\ast\phi_{\cqg/\fdg}}A\\
\TT_{\cngp/\cqg}@>>>\TT_{\cngp/\fdg}@>>>\upsilon^\ast \TT_{\cqg/\fdg} @>{+1}>>\\
\end{CD}
\eeq
Composing the cosection $\sigma_0:\Ob_{\cngp/\fdg}\to\sO_{\cngp}$ (cf. \eqref{sig-o}) with $H^1(\EE_{\cngp/\cqg})\to
H^1(\EE_{\cngp/\fdg})$, we obtain
$$\sigma'_0:=\Ob_{\cngp/\cqg}\lra \sO_{\cngp}.
$$
Arguing similar to Proposition \ref{degQ}, one sees that the degeneracy loci of $\sigma'_0$ equals $\cqg\sub\cngp$.

Now let $U:=\cngp-\cqg$; it is open in $\cngp$, and both
$\sigma_0$ and $\sigma_0'$ are surjective on $U$. By the octahedral axiom, we have a diagram 
\beq\label{Oct}
\begin{CD}
\sO_U[-1]@>{=}>>\sO_U[-1]\\
@AAA@AAA\\
\EE_{\cngp/\cqg}|_U@>>>\EE_{\cngp/\fdg}|_U@>>>\upsilon^\ast \EE_{\cqg/\fdg}|_U@>{+1}>>.\\
@AA{\chi_Q}A@AA{\chi}A@AA{||}A\\
 \EE'_{U/\cqg}@>>>\EE'_{U/\fdg}@>>>\upsilon^\ast \EE_{\cqg/\fdg}|_U@>{+1}>>\\
\end{CD}
\eeq
where all rows and columns are distinguished triangles, and the two vertical rows to $\sO_U[-1]$
are induced by $\sigma_0$ and $\sigma_0'$, respectively.

\begin{lemm}\lab{factor}
There are perfect relative obstruction theories $\phi'_{U/\cqg}$ of $U/\cqg$ and $\phi'_{U/\fdg}$ of $U/\fdg$
that fit into a compatible diagram
\beq\label{triangle}
\begin{CD}
\EE'_{U/\cqg}@>{\theta_E}>>\EE'_{U/\fdg}@>{h\circ \chi}>>\upsilon^\ast \EE_{\cqg/\fdg}@>{+1}>>.\\
@AA{\phi'_{U/\cqg}}A@AA{\phi'_{U/\fdg}}A@AA{\upsilon^\ast\phi^{\leq 1}_{\cqg/\fdg}|_U}A\\
\TT^{\leq 1}_{U/\cqg}@>{\theta}>>\TT^{\leq 1}_{U/\fdg}@>>>(\upsilon^\ast \TT_{\cqg/\fdg} )|_U^{\leq 1}@>{+1}>>\\
\end{CD}
\eeq \end{lemm}

\begin{proof}
Applying the truncation functor $\tau_{\leq 1}$ to $\phi_{\cngp/\fdg}|_U$, we obtain
$$\phi^{\leq 1}_{\cngp/\fdg}|_U:\TT^{\leq 1}_{\cngp/\fdg}|_U\lra \EE_{\cngp/\fdg}|_U.$$
Then the commutative diagram 
$$
\begin{CD}
\TT^{\leq 1}_{\cngp/\fdg}|_U@>{\phi^{\leq 1}_{\cngp/\fdg}|_U}>> \EE_{\cngp/\fdg}|_U\\
@VVV @VVV\\
H^1(\TT_{\cngp/\fdg}|_U)@>>>\Ob_{\cngp/\fdg}|_U
\end{CD}
$$
implies that the composition of $\phi^{\leq 1}_{\cngp/\fdg}|_U$ with $\EE_{\cngp/\fdg}|_U\to \sO_U[-1]$ vanishes.
Hence $\phi^{\leq 1}_{\cngp/\fdg}|_U=\chi\circ \phi'_{U/\fdg}$ for some
$$\phi_{U/\fdg}:\TT^{\leq 1}_{U/\fdg}\lra \upsilon^\ast \EE_{\cqg/\fdg}.
$$
It is direct to check $\phi_{U/\fdg}$ is a perfect obstruction theory, and the middle square of
the diagram (\ref{triang}) commutes. By similar reason the $\tau_{\leq 1}$ truncation of 
$\phi_{\cngp/\cqg}|_U$,
$$\phi_{\cngp/\cqg}|_U:\TT^{\leq 1}_{U/\cqg}\lra \EE_{\cngp/\cqg}|_U,
$$
has its composition with $\EE_{\cngp/\cqg}|_U\to \sO_{U}[-1]$ vanishes and lifts to a
$$\phi_{U/\cqg}':\TT^{\leq 1}_{U/\cqg}\lra \EE_{U/\fdg}$$
such that 
$\phi_{\cngp/\cqg}|_U=\chi_Q \circ \phi'_{U/\cqg}$.
The map $\Delta:=\theta_E\circ\phi_{U/\cqg}'-\phi'_{U/\fdg}\circ\theta$ in (\ref{triangle})
thus  satisfies 
$\chi\circ\Delta=0$, hence $\Delta$ factors throught a morphism 
$\sO_U[-2]\lra \EE'_{U\fdg}$, which imples $\Delta=0$; this is because when applying the truncation functor $\tau_{\leq 1}$ one 
obtains  $\tau_{\leq 1}(\Delta)=\Delta$ and $\tau_{\leq 1}(\sO_U[-2])=0$.
\end{proof}

We now quote the virtual pull-back construction of Manolache in \cite{Cristina}. 
First the compatibilty diagram (\ref{triang}) fits into the condition two in the construction of Manolache in \cite{Cristina}. 
Let $\bC_{\cngp/\cqg}$ be the intrinsic normal cone of $\cngp$ relative to $\cqg$ and let
$i:\bC_{\cngp/\cqg}\to h^1/h^0(\EE_{\cngp/\cqg})$ be the inclusion by the relative perfect obstruction theory (\ref{relll}). 
We %denote $G:= h^1/h^0(\EE_{\cngp/\cqg})$ and 
let $G'$ be the kernel of the morphism of bundle staks 
$$G'=\ker\{\sigma_0': G:= h^1/h^0(\EE_{\cngp/\cqg})\lra \CC_{\cngp}\},
$$
where the arrow in the bracket is induced by the cosection $\sigma'_0$, 
%($\CC_{\cngp}$ is the trivial line bundle over $\cngp$),
and the kernel is defined in \eqref{cone-stack}. By the Cosection lemma in \cite{KL} and Lemma \ref{factor}, we have
\beq\label{iG}
i(\bC_{\cngp/\cqg})\sub i(\gamma^\ast \bC_{{\fC}/\fdg})\sub G'.
\eeq
Note that here the virtual rank of the bundle stack $G$ is zero.
 
We generalize the construction in \cite{Cristina} and give a virtual pullback morphism of cosection localized classes
$$i^!_{G'}: A\lsta\cqg\lra A\lsta\cqg
$$
defined as the composite of 
\begin{equation}\lab{com}
A\lsta\cqg\mapright{\zeta}A\lsta\bC_{\cngp/\cqg}\mapright{i\lsta} A\lsta G'
\mapright{0^!_{\sigma_0,\mathrm{loc}}} A\lsta\cQ,
\end{equation}
where $0^!_{\sigma_0,\mathrm{loc}}$ is the localized Gysin map defined in \cite{KL}, and $\zeta$
defined by first sending a cycle
$\sum n_i[V_i]$ to $\sum n_i [\bC_{V_i\times_\cqg \cngp/V_i}]$, and then 
descending it to cycle class group. Note that $i\lsta$ maps to $A\lsta G'$ is due to \eqref{iG}.
 
Following the same argument as in Corollary 4 in \cite{Cristina}, we have
\begin{lemm}\lab{pullback}
$$i^!_{G'}([\cqg]^{vir})=[\barM_g(\nq,d)^p]\virt_{\sigma_0}\in A_0 \cQ.
$$
\end{lemm}

\begin{proof}
One needs to show that the KKP's deformation to normal cone lies inside the kernel of the cosection, 
which follows from the Lemma \ref{weakKKP} in Appendix and Lemma \ref{factor}.
\end{proof}
 
We now prove our main Theorem.

\begin{theo}\label{formal}
%$h=\chi(C,L^{\otimes 5})=5d+1-g$ then
We have 
$$N_g(d)_\Pf^p=N_g(d)_{K_{N_Q}}=(-1)^{5d+1-g}\cdot N_g(d)_Q.
$$
\end{theo}
   
\begin{proof}
We first compute the degree of the zero cycle $i^!_{G'}([\xi])\in A_0\cQ$, where $\xi$ is any closed point in $\cQ$. 
Let $\xi$ be $[u,C]\in \cQ$. We denote 
$$V_1=H^0(C,u^\ast \sO(5)), \ V_2=H^0(C,u^\ast\sO(-5)\otimes \omega_{C})\cong V_1\dual, \ \ V=V_1\oplus V_2.
$$
It is direct to check ($v$ is defined in diagram \eqref{v-v}) 
$$\upsilon^{-1}\xi:=\cngp\times_\cqg \xi\cong V, \quad G|_{\upsilon^{-1}\xi}\cong [V\times V\dual/V],
$$
where the action of $V$ on $V\times V\dual$ is via the zero homomorphism $0:V\to V\times V\dual$.
One also checks that the cosection $\sigma_0$ restricted to $\upsilon^{-1}\xi$ is   
induced by $$\sigma_\xi:V\times V\dual = (V_1\oplus V_2)\times (V_1\dual\oplus V_2\dual)\lra \CC,
$$
given by dual parings $V_i\times V_i\dual\to \CC$. %It is also direct to check
 
Applying the composition (\ref{com}) step by step, from 
$$\zeta ([\xi])=[\bC_{V/\xi}]\in A\lsta(G'|_\xi),
$$
 %by the following commutative diagram 
 %$$
 %\begin{CD}
 %A\lsta(G'|_\xi)@>{0^!_{\sigma_\xi,\mathrm{loc}}}>> A\lsta(\upsilon^{-1}\xi)\\
 %@VVV @VVV\\
 %A\lsta(G')@>{0^!_{\sigma_0,\mathrm{loc}}}>> A\lsta(\cngp) 
 %\end{CD}
 %$$
we have
$$i^!_{G'}([\xi])=0^!_{\sigma_\xi,\mathrm{loc}}(\bC_{V/\xi})=(-1)^{\rank V}[\xi]=(-1)^{5d+1-g}[\xi]\in A_0\cQ.
$$
Here the second equality follows from 
$$\bC_{V/\xi}=[V\times 0/V]\sub [V\times V\dual/V]=G|_{v^{-1}\xi},
$$
and \cite[Example 2.9]{KL}. Finally $\rank V=5d+1-g$ by Riemann-Roch theorem. 

Taking degree,
$$\deg \ i^!_{G'}([\xi])=(-1)^{5d+1-g}.
$$
Since both $[\cqg]^{vir}$ and $[\barM_g(\nq,d)^P]\virt_{\sigma_0}$ in Lemma 
\ref{pullback} are of zero dimensions, 
taking degrees we obtain
$$\deg\ [\barM_g(\nq,d)^P]\virt_{\sigma_0} = \deg\ i^!_{G'}([\xi])\cdot \deg\ [\cqg]^{vir}=(-1)^{5d+1-g} N_g(d)_Q.
$$
This proves the second identity in the statement of the theorem. The first identity is Theorem \ref{thm5.3}.
 \end{proof}

\section{appendix}

\def\boldY{{\mathbf Y}}
\def\boldX{{\mathbf X}}
\def\boldE{{\mathbf E}}
\def\boldF{{\mathbf F}}
\def\fk{{\mathfrak k}}
\def\ufl{^{\text{flat}}}
\def\ellip{\text{ell}}
\def\gst{\text{gst}}
\def\lred{_{\text{red}}}
\def\cA{{\mathcal A}}

We recall some useful facts known to the experts.

\subsection{Kresch-Kim-Pantev's construction}
Let $S$ be a stack.

\noindent
{\bf Convention}. {\sl For a complex (derived object) $\GG$ on $S$, we denote $\GG(k)$ without further commenting to be $$\GG(k)\defeq p_S\sta \GG\otimes p_\Po\sta \sO(k);
$$
further, whenever we see a complex over $S$ appearing in a sequence involving complexes over $S\times\Po$,
we understand the complex as its pull-back from $S$.}
\vsp

\begin{defi}\lab{mapping cone}
Let
$\EE_1\mapright{b} \EE_2\lra \EE_3\mapright{+1}$
be a distinguished triangle of objects in $\bD(S)$ whose cohomologies concentrated at non-positive degrees. Assume $\EE_1$ is of amplitude in $[-1,\infty]$.
Let $[x,y]$ be the homogeneous coordinates of $\Po$, and let 
$$\bar{b}: \EE_1(-1)\to \EE_1\oplus \EE_2
$$ 
be defined by
$(x\cdot 1, y\cdot b)$. We form the mapping cone $c(\ti b)$ of $\ti b$, which fits into the distinguished triangle
$$\EE_1(-1)\mapright{\bar{b}} \EE_1\oplus \EE_2\lra c(\ti b)\mapright{+1}.
$$
Applying the $h^1/h^0$ construction to $c(\bar b)\dual$, we obtain $h^1/h^0(c(\bar b)\dual)$,
which is a cone-stack over $S\times\Po$  \cite{BF}.
Following \cite{Kresch2} we call it the deformation
of $h^1/h^0(\EE_2\dual)$ to $h^1/h^0(\EE_1\dual)\times_S h^1/h^0(\EE_3\dual)$.
\end{defi}

%We comment that we will reserve tilde ``$\ti b$'' for the construction in this Definition.
Let $i:X\to Y$ and $j:Y\to Z$ be morphims of relative Deligne-Mumford type, between stacks. 
Let 
\beq\label{tri-app}
i^\ast\LL_{Y/Z}\mapright{\beta}\LL_{X/Z}\lra \LL_{X/Y} \mapright{+1}
\eeq
 be the
induced distinguished triangle of cotangent complexes. 
We quote the main theorem of \cite{Kresch2}. 

\begin{prop}\cite{Kresch2}\label{prop6.2}
We have a natural isomorphism
$$N_{X\times\Po/M^\circ_{Y/Z}}\cong h^1/h^0(c(\ti \beta)\dual).
$$
\end{prop}
 
  Now we stated a truncated version which is dual to Definition \ref{mapping cone}.
  
\begin{lemm}\label{weakKKP}
%Let the notation be as in Proposition \ref{prop6.2},
Let
$$\TT^{\leq 1}_{X/Y}\lra \TT^{\leq 1}_{X/Z}\mapright{k} i^\ast \TT^{\leq 1}_{Y/Z}
$$
%$$\TT_{X/Y}\lra \TT_{X/Z}\lra i^\ast \TT_{Y/Z}\mapright{+1}
%$$
be the truncation by $\tau_{\leq 1}$ of the dual of the distinguished triangle \eqref{tri-app} .
(It is not a distinguished triangle.)
Let $c_0(\ti k)$ be defined by making 
$$c_0(\ti k)\lra  i^\ast \TT^{\leq 1}_{Y/Z}\oplus  \TT^{\leq 1}_{X/Z} \mapright{\ti k} i^\ast \TT^{\leq 1}_{Y/Z}\otimes \sO_\Po(1)
$$
a distinguished triangle, where $\ti k=(x,y\cdot k)$ as in Definition \ref{mapping cone}.
Then there is a natural isomorphism
$$h^1/h^0(c(\ti \beta)\dual)\cong h^1/h^0(c_0(\ti k)).
$$
\end{lemm}

\begin{proof}
Using simplicial resolution of Illusied, we can represent $\i^\ast\LL_{Y/Z}$ and $\LL_{X/Z}$ by perfect complex (over $X$ globally) of amplitude $[-\infty,0]$ and represent $\beta:i^\ast\LL_{Y/Z}\to\LL_{X/Z}$ by
a homomorphism of between these two complexes. 
From this it is direct to show that the canonical morphism
\beq\label{cccc}c_0(\ti k)\lra c(\ti \beta)\dual 
\eeq
induces isomorphisms on $H^1$ and $H^0$ of the two complexes in \eqref{cccc}. Hence their truncations by 
$\tau_{\leq 1}$ are isomorphic under this arrow, which shows that
the cone-stacks of the $h^1/h^0$ constructions of the two complexes in \eqref{cccc} 
are isomorphic under the arrow induced by \eqref{cccc}.
\end{proof}

\subsection{Application}

We recall the rational equivalence inside the deformations of ambient conestacks  constructed by
Kim-Kresch-Pantev \cite{Kresch2}.
%(The main part is the intrinsic deformation to normal cone; this is a supplementary
%to the main part of his paper.)

Let ${Z}$ be an Artin stack, locally of finite type and of pure-dimension. Let ${Y}$ be a stack and ${Y}\to {{Z}}$ be a morphism of relative Deligne-Mumford type.in the derived category of coherent sheaves on ${X}$.
Let $\EE\dual$ (resp. $\FF\dual$, $\VV\dual$) be a perfect relative obstruction theory of ${X}/{Z}$ (resp. ${Y}/{Z}$, $X/Y$).
\begin{defi}
We say $\FF$ and $\EE$ are \textit{truncated-compatible} (verses $(V,s)$)
if there exists a commutative diagram
\beq\lab{compatibility}
\begin{CD}
 \VV@>>> \EE@>g>>\FF|_X@>>+1>\\ % \FF|_{{X}}[1]\\
@AAA@AAA@AAA  \\
 \TT_{X/Y}^{\leq 1}@>>> \TT^{\leq 1}_{{X}/{Z}} @>{k}>> \TT_{{Y}/Z}|_X^{\leq -1} % \LL_{{Y}/{{Z}}}|_{{X}}[1].
\end{CD}
\eeq
such that its top row is a distinguished triangle, and its bottom row is
the first line in Lemma \ref{weakKKP}. 
\end{defi}

Accordingly, the morphisms $g$ and $k$ in \eqref{compatibility} induces homomorphism $\ti g$ and $\ti k$
that fit into a homomorphism of distinguished triangle's
$$
\begin{CD}
c_0(\ti g)@>>> \FF|_{{X}}\oplus \EE @>\ti g>>\FF|_X(1) @>>+1>\\
@AAA@AAA@AAA \\
c_0(\ti k)@>>> \TT^{\leq 1}_{{Y}/{{Z}}}|_{{X}}\oplus \TT^{\leq 1}_{{X}/{{Z}}} @>\ti k>>\TT^{\leq 1}_{Y/Z}|_X\otimes \sO_\Po(1)
 @>>+1> ,
\end{CD}
$$
where $c_0(\ti g)$ is to make the first row a distinguished triangle as $c_0(\ti k)$ did in Lemma \ref{weakKKP}.
We let $M^0_{{Y}/Z}$ be the deformation of $Z$ to the normal cone $\bC_{Y/Z}$; let
$\bC_{{X}\times \Po /M^0_{{Y}/Z}}$ be the normal cone to $X\times \Po$ in $M^0_{{Y}/Z}$,
and let $N_{{X}\times \Po /M^0_{{Y}/{{Z}}}}$ be the normal sheaf of ${X}\times \Po$ in $M^0_{{Y}/{{Z}}}$.
By the functoriality of the $h^1/h^0$ construction, we have %$\cD:=\bC_{{X}\times \Po /M^0_{{Y}/{{Z}}}}$
\begin{equation}\lab{inclusion}
\cD:=C_{{X}\times \Po /M^0_{{Y}/Z}}
\sub  N_{{X}\times \Po /M^0_{{Y}/{{Z}}}}\cong h^1/h^0(c_0(\ti k)),
\eeq
where the isomorphism is proved in \cite{Kresch2} and Lemma \ref{weakKKP}.  We also have the inclusion
\beq\lab{inc-2}h^1/h^0(c_0(\ti k)) \sub h^1/h^0(c_0(\ti g))\cong h^1/h^0(c(\overline{g\dual})\dual),
\eeq
where $h^1/h^0(c(\overline{g\dual})\dual)$ is the deformation of 
$h^1/h^0(\EE)$ to $h^1/h^0(\FF|_X)\times_X h^1/h^0(\VV)$ as in
Definition \ref{mapping cone}. This shows that the truncated compatibility (\ref{compatibility}) is sufficient to apply Kresch-Kim-Pantev construction of rational equivalence.

\subsection{Obstruction class assignments}

Assume there is a smooth morphism of Artin stacks $H\to W$. Suppose $T\sub T'$ is a pair of affine schemes 
such that $J:=I_{T/T'}$ and $J^2=0$. Fix a morphism $T'\to \fdg$, which pulls back $\pi_\fdg:\cC_\fdg\to \fdg$ to
$\pi_T:\cC_T\to T$ and $\pi_{T'}:\cC_{T'}\to T'$. Assume there is a commutative diagram
\beq\label{lift-general}
\begin{CD}
\cC_T@>{\ee}>>H \\
@VVV@VVV\\
\cC_{T'}@>>>W.
\end{CD}
\eeq

Since the ideal sheaf of $\cC_T\sub \cC_{T'}$ is $\pi_{T'}^\ast J$, it is a square zero extension. 
We denote $V_T:=\ee^\ast\Omega_{H/W}\dual$ then $V_T$ is a locally free sheaf over $\cC_T$.
The diagram (\ref{lift-general}) provides a morphism
\beq\label{ob11}
V_T\dual\cong \ee^\ast\LL^{}_{H/W}\lra \LL^{}_{\cC_T/\cC_{T'}} = \pi_T^\ast\LL^{}_{T/T'}
\lra \LL_{\cC_T/\cC_{T'}}^{\geq -1}=   \pi_T^\ast J[1],
\eeq
(here $\ee^\ast$ denotes derived pull back)
which defines an element 
$$\omega(\ee,H,W)\in \Ext^1_{\cC_T}(V_T\dual, \pi_T^\ast J)\cong H^1(\cC_T,V_T\otimes \pi_T^\ast J).
$$ 

\begin{lemm}\label{ob-general}
$\omega(\ee,H,W)=0$ if and only if the diagram (\ref{lift-general}) admits a lifting $\cC_{T'}\to H$ that commutes with the diagram.
\end{lemm}

\begin{proof}
We form the diagram
\beq\label{Illusie}
\begin{CD}
X_0:= \cC_T@>{\ii}>> X:=\cC_{T'}\\
@VV{\bar\ee}V\\
Y_0:=H\times_{W} \cC_T @>{\jj}>> Y:=H\times_{W} \cC_{T'}\\
@VV{\Delta}V @VVV\\
\cC_T @>{\sub}>>  S:=\cC_{T'}
\end{CD}
\eeq
where $\ii$ and $\jj$ are extensions over $S$. By construction, the associated homomorphism of sheaves
$$v:\bar\ee^\ast I_{Y_0/Y}\lra I_{X_0/X}=\pi_T^\ast J
$$
is an isomorphism.
If $\bar\ee$ lifts to a $\cC_{T'}$-morphism $\bar\ee':X\to Y$, then the $X\to\cC_{T'}$ is
an isomorphism.

Following the steps in the proof of \cite[Thm 2.1.7]{Illusie}, the obstruction to the existence of such $\bar\ee'$ (in the notation of 
\cite{Illusie}) are constructed as follows. First one has a sequence of cotangent complexes
\beq\label{seqq}
\LL_{X_0/Y_0}[-1]\lra \bar\ee^\ast \LL_{Y_0/S} \lra \bar\ee^\ast\LL_{Y_0/Y}\lra \bar\ee^\ast\LL_{Y_0/Y}^{\geq -1}\lra \pi_T^\ast J[1],
\eeq
where the first (left) morphism comes from the triple $X_0\to Y_0\to S$;
the middle morphism is induced by $\LL_{Y_0/S}\to \LL_{Y_0/Y}$.

Using $\LL_{X_0/Y_0}=V_T\dual[1]$, this sequence associates an element
$$\omega(\bar\ee, \jj)\in \Ext^2_{X_0}(\LL_{X_0/Y_0}, \pi_T^\ast J)=\Ext^1_{X_0}(V_T\dual,\pi_T^\ast J)=H^1(\cC_T,V_T\otimes \pi_T^\ast J).
$$
The argument in \cite[Thm 2.1.7]{Illusie} shows 
that $\omega(\bar\ee,\jj)=0$ if and only if a lift $\bar\ee':X\to Y$ exists in the diagram (\ref{Illusie}). 
Such lift exists if and only if a lift $\ee':\cC_{T'}\to H$ exists in the diagram (\ref{lift-general}). Hence we only need to verify  that 
$\omega(\bar\ee,\jj)=\omega(\ee,H,W)$.

To this end, we verify the commutativity of the following diagram
\beq\lab{GGG}
\begin{CD}
\LL_{X_0/Y_0}[-1]@>>>\bar\ee^\ast\LL_{Y_0/S}@>>>\bar\ee^\ast \LL_{Y_0/Y}\\
@VV{\cong}V @AAA @AA{\cong}A\\
\bar\ee^\ast\LL_{Y_0/X_0}@<{\cong}<<\bar\ee^\ast \jj^\ast\LL_{Y/S}@>>>\LL_{X_0/S},
\end{CD}
\eeq
where the first vertical arrow is an isomorphism because $\LL_{X_0}=0$;
the left square is commutative because 
the canonical $\LL_{Y_0/S}\to\LL_{Y_0/X_0}$ induces a
$\bar\ee^\ast\LL_{Y_0/S}\to\bar\ee^\ast\LL_{Y_0/X_0}$ that splits the left square into two commutative triangles of cotangent complexes;
the third vertical arrow is composing $\LL_{X_0/S}\cong \bar\ee^\ast\Delta^\ast\LL_{X_0/S}$
with the isomorphism $\Delta^\ast\LL_{X_0/S}\mapright{\cong} \LL_{Y_0/Y}$.
The right square is commutative because one has a canonical pullback $\bar\ee^\ast\LL_{Y_0/S}\to\LL_{X_0/S}$ 
and a commutative diagram
$$
\begin{CD}
\bar\ee^\ast\LL_{Y_0/S}@>>>\bar\ee^\ast\LL_{Y_0/Y}\\
@VVV @AA{\cong}A\\
\LL_{X_0/S}@<{\cong}<<\bar\ee^\ast\Delta^\ast\LL_{X_0/S}. 
\end{CD}
$$
The upper and lower rows of the diagram (\ref{GGG}) are repectively  sequence (\ref{ob11}) and (\ref{seqq}).  
Thus the commutative diagram (\ref{GGG}) implies  $\omega(\bar\ee,\jj)=\omega(\ee,H,W)$.
\end{proof}

\bibliographystyle{amsplain}

\end{document}